\documentclass[10pt,a4paper]{amsart}%

\usepackage{MyCommA}
\usepackage[final]{pdfpages}
\usepackage{pdfsync}
\newcommand{\simpAr}[2][r]{%
\ar@{}[#1]|-*[@]_{#2}%
}
\renewcommand{\colorlinks}{true}
\renewcommand{\linkcolor}{lblue}
\renewcommand{\citecolor}{lblue}
\renewcommand{\urlcolor}{dblue}
\renewcommand{\linkbordercolor}{red}
\renewcommand{\citebordercolor}{green}
\renewcommand{\urlbordercolor}{cyan}
\hypersetup{colorlinks=\colorlinks,linkbordercolor=\linkbordercolor, linkcolor=\linkcolor, citecolor=\citecolor, urlcolor=\urlcolor,urlbordercolor=\urlbordercolor,citebordercolor=\citebordercolor}%

\definecolor{vert}{rgb}{0,0.6,0.2}

\usepackage[color, leftbars]{changebar}

\def\WF{\mathrm{WF}}

\newcommand{\Char}{\mathrm{SS}}
\newcommand{\scD}{\mathscr{D}}

\renewcommand{\bfX}{\mathbf{X}}

\long\def\change#1{
 #1}

\newcommand{\Bf}{\overline{B}}
\newcommand{\CN}{\text{CN}}
\newcommand{\norm}[1]{\left|\!\left|#1\right|\!\right|}
\newcommand{\abs}[1]{\left|#1\right|}
\newcommand{\Cexp}{\cC^{\mathrm{exp}}}

\newcommand{\cCexp}{\cC^{\mathrm{exp}}}

\renewcommand{\alg}{\mathrm{alg}}

\def\Supp{\operatorname{Supp}}

\def\11{{\mathbf 1}}
\def\AA{{\mathbb A}}

\def\CC{{\mathbb C}}

\def\NN{{\mathbb N}}

\def\QQ{{\mathbb Q}}
\def\RR{{\mathbb R}}

\def\ZZ{{\mathbb Z}}

\def\cC{{\mathscr C}}
\def\cD{{\mathcal D}}

\def\cF{{\mathcal F}}
\def\cG{{\mathcal G}}

\def\cI{{\mathcal I}}

\def\cO{{\mathcal O}}

\def\cS{{\mathcal S}}

\def\cU{{\mathcal U}}

\def\cW{{\mathcal W}}

\theoremstyle{plain}

\numberwithin{equation}{subsection}

  {\par\medskip\noindent #1\par\begingroup%
    \advance\leftskip by 1em\advance\rightskip by 1em}%
  {\par\endgroup}

\newcommand{\Jac}{\operatorname{Jac}}

\newcommand\blue[1]{{{#1}}}

\newcommand{\zerodel}{.\kern-\nulldelimiterspace}

\newcommand{\la}{\left \langle}
\newcommand{\ra}{\right \rangle}
\newcommand{\lp}{\left (}
\newcommand{\rp}{\right )}

\begin{document}

\author[Aizenbud]
{Avraham Aizenbud}
\address{Faculty of Mathematical Sciences,
Weizmann Institute of Science,
Rehovot, Israel}
\email{aizenr@gmail.com}
\urladdr{http://aizenbud.org}

\author[Cluckers]
{Raf Cluckers}
\address{Univ.~Lille, CNRS,  UMR 8524 - Laboratoire Paul Painlev\'e, F-59000 Lille, France, and
KU Leuven, Department of Mathematics, B-3001 Leu\-ven, Bel\-gium}
\email{Raf.Cluckers@univ-lille.fr}
\urladdr{http://rcluckers.perso.math.cnrs.fr/}


\author[Raibaut]{Michel Raibaut}
\address{Laboratoire de Math\'ematiques\\
Universit\'e Savoie Mont Blanc, B\^atiment Chablais, Campus Scientifique, Le Bourget du Lac, 73376 Cedex, France}
\email{Michel.Raibaut@univ-smb.fr}
\urladdr{www.lama.univ-savoie.fr/~raibaut/}

\author[Servi]{Tamara Servi}
\address{Institut de Math\'ematiques de Jussieu - Paris Rive Gauche
Universit\'{e} Paris Cit\'{e} and Sorbonne Universit\'{e}, CNRS, IMJ-PRG, F-75013 Paris, France}
\email{Tamara.Servi@imj-prg.fr}
\urladdr{https://tamaraservi.github.io/}


\subjclass[2020]{Primary 46F99, 32S40, 32S60, 35Nxx, 03C10 Secondary 22E45, 26B15, 14F10, 32C38,  58J10.} 

\keywords{Distributions of C$^{{\mathrm{exp}}}$-class, embedded resolution of subanalytic functions, analytically holonomic D-modules, Fourier transforms of tempered distributions, micro-local analysis, wave front sets, local structure theorem for distributions}

\thanks{The authors would like to thank Daniel Miller, Claude Sabbah and Mathias Stout for several interesting discussions related to the paper and the referee for very valuable input, in particular on Proposition \ref{cor:base-case-u}. The author R.~C.~was partially supported by KU Leuven IF C16/23/010 and acknowledges the support of the CDP C2EMPI, as well as the French State under the France-2030 programme, the University of Lille, the Initiative of Excellence of the University of Lille, the European Metropolis of Lille for their funding and support of the R-CDP-24-004-C2EMPI project.}

%
%
%
%
%
%
%
%


\title[Holonomicity of C$^{{\mathrm{exp}}}$-class distributions]{Analytic holonomicity of real C$^{{\mathrm{exp}}}$-class distributions}
\begin{abstract}
	We introduce a notion of distributions on $\RR^n$, called distributions of $\cCexp$-class, based on wavelet transforms of distributions and the theory from \cite{CCMRS} about $\cCexp$-class functions. We prove that the framework of $\cCexp$-class distributions is closed under natural operations, like push-forward, pull-back, derivation and \blue{anti-derivation}, and, in the tempered case, Fourier transforms. Our main result is the (real analytic) holonomicity of all distributions of $\cCexp$-class.
\end{abstract}

\maketitle

\section{Introduction}


In this paper, we introduce a notion of distributions on an open set $U\subset \RR^n$, called distributions of $\cCexp$-class, see Definition \ref{def:distrCexp}, and we prove their holonomicity, see Theorem \ref{mainthm}.

The definition of $\cCexp$-class distributions is based on both continuous wavelet transforms of distributions, 
and, the theory from \cite{CCMRS} about $\cCexp$-class functions. More precisely, the continuous wavelet transform (with respect to a well-chosen mother wavelet) of a distribution $u$ is a function on $U$ which is required to be a function of $\cCexp$-class in order for $u$ to be a distribution of $\cCexp$-class. The mother wavelet is taken to be a $C^\infty$-function of $\cCexp$-class and with compact support, see Definitions \ref{notation:motherwavelet} and \ref{defn:wevelet}. The functions of $\cCexp$-class are introduced in \cite{CCMRS}; they form a tame class of functions which
contains all subanalytic functions (namely, from the o-minimal structure $\RR_{\rm an}$) and have the property that parametric integrals and parametric Fourier transforms are still of $\cCexp$-class, see Definition \ref{def:Cexp} and Remark \ref{rem:Cexp}.

In order to show our main theorem, we need to develop the basic theory of distributions of $\cCexp$-class.
We prove that the framework of $\cCexp$-class distributions is closed under natural operations, like push-forwards along proper subanalytic maps, pull-backs along suitable subanalytic maps, derivation, and \blue{anti-derivation}, see Sections \ref{sec:uCexp} and \ref{sec:antider}. 
If a distribution on $\RR^n$ is of $\cCexp$-class and furthermore tempered, then also its Fourier transform is of $\cCexp$-class.
These basic results, together with the theory of \cite{CCMRS} and a new variant of resolution of singularities (or rather, an alteration result) for subanalytic functions (see Theorem \ref{thm:res:def}), allow us to prove the following result, our main result.
The notion of holonomicity that we use in this paper is the real analytic holonomicity from Definition \ref{def:holonomic}.

\begin{introtheorem}[Proof in Section \ref{sec:holon}] \label{mainthm}
	Any $\cCexp$-class distribution $u$ on an open set $U\subset \RR^n$ is holonomic.
\end{introtheorem}

Roughly speaking, holonomicity of $u$ means it satisfies a large enough system of linear PDEs with real analytic coefficients. This notion of holonomicity (from Definition \ref{def:holonomic}) is more general than algebraic holonomicity, and stands for real analytic holonomicity which is based on D-modules with real analytic functions (see \cite[page 298]{Bjork} and Definition \ref{def:holonomic}).  Note that this real-analytic version of holonomicity is not automatically preserved under Fourier transform of tempered distributions, see Example \ref{ex}.

The work in this paper constitutes the real counterpart of the non-archimedean results of \cite{CHLR} and \cite{AizC}, where a rich class of distributions on $\QQ_p^n$ was introduced and their properties studied, which in particular addressed a question from \cite{AizDr} about wave front holonomicity.

Our framework of $\Cexp$-class distributions is stable under several of the usual operations on distributions and also under taking Fourier transforms (in the tempered case). This is similar to the situation of algebraic holonomic D-modules, which are known to be stable under Fourier transforms, see Example 2 and Corollary 3.4 of \cite{Bernstein:D-mod}.


 \medskip
We also obtain a local structure theorem of $\cCexp$-class distributions which is a key ingredient to prove Theorem \ref{mainthm}. 

\begin{introtheorem}[Proof in Section \ref{sec:antider}]
\label{thm:derivative}
Let $u$ be a $\cCexp$-class distribution on an open set $U \subset \RR^n$. Suppose that $u$ has compact support. Then there \blue{exist} $k \in \NN^n$ and a $\cCexp$-class continuous function $g$ on $U$ such that the iterated partial derivative $\frac{\partial^k}{\partial x^k} g$ equals $u$. Here, we use multi-index notation, and the partial derivatives are taken in the sense of distributions.
\end{introtheorem}

The following is a corollary of Theorem 
\ref{thm:derivative} and some properties of $\cCexp$-class functions.

\begin{introcorollary}[Proof in Section \ref{sec:antider}]\label{thm:gen.sm}
Let $u$ be a distribution of $\cCexp$-class on an open set $U\subset \R^n$. Then $u$ is analytic on a dense open subset $U'$ of $U$.
If furthermore the support of $u$ is compact, then we can take $U'$ to be subanalytic.
\end{introcorollary}


\subsection{Motivation and context}
In the non-archimedean case, the notion of $\cCexp$-class distributions, introduced in \cite{CHLR} and studied further in \cite{AizC}, is a natural framework to analyze distributions.
It is closed under all natural operations, it contains many useful examples of distributions,
and it allows a uniform treatment for all non-archimedean local fields. In \cite{AizC}, still in the non-archimedean case, it is shown that each $\cCexp$-class distribution is wave-front holonomic, a notion introduced in \cite{AizDr} to encompass the lack of a deeper holonomicity notion in the non-archimedean case.
In the archimedean case, the notions of (analytic) holonomicity are based on PDE's, which are not applicable in the non-archimedean case.

Our newly introduced  $\cCexp$-class distributions in the archimedean case are very similar to the one of \cite{CHLR}, and have similar properties, like being stable under Fourier transform and appropriate pull-backs and push-forwards. They are introduced in each setting via wavelet transforms.
Theorem \ref{mainthm} allows to propagate all the good properties of holonomic distributions to $\cCexp$-class distributions.

\subsection{Sketch of the proof and difference with the non-archimedean case}
The definition of the $\cCexp$-class of distributions as well as the proofs of its stability under natural operations are somewhat similar to those in the non-archimedean case studied in \cite{CHLR}.
The main new challenge is to prove Theorem \ref{mainthm}.
The arguments from the non-archimedean case are not directly applicable in the archimedean case for the following reasons. We will give precise notation and definitions later below.

\begin{itemize}
  \item Distributions supported on a submanifold are not the same as distributions defined on this manifold.
  \item A  $\cCexp$-class function involves not only subanalytic functions and their exponentials, but also some new transcendental 
  functions, namely, some basic integrals of subanalytic functions and their exponentials, see Definition \ref{def: constructible}.
\end{itemize}
As a new ingredient, we use anti-derivatives in order to reduce to the case of locally $L^1$ functions of $\cCexp$-class.
In more detail:
\begin{enumerate}
	\item For any $i \in \{1,\dots,n\}$, we define an anti-derivative operator $A_i$ on the space of distributions on $\R^n$
		$$A_{i}:\mathscr{D}'(\R^n)\to \mathscr{D}'(\R^n)$$
		satisfying
	  $$\frac{\partial}{\partial x_i} A_{i}(u)= u.$$
	  We do it in the following way (for simplicity here we describe the $1$-dimensional case): fix a $\cCexp$-class test function $\rho$ of integral $1$ and define for any test function $f$ on $\RR^n$ (namely, a $C^\infty$ function with compact support) 
	  $$\la A(u),f \ra = \la u,B\left(f-\rho\int_{-\infty}^{\blue{+\infty}} f\right)\ra, $$
	  where $B(g)$ for a test function $g$ is the function on $\R$ given by
	  $$B(g)(x)=\blue{-}\int_{-\infty}^x g(y)\mathrm{d}y,$$
this is treated in Section \ref{sec:antider}.  

  \item We prove that $A_i$ preserves the $\cCexp$-class of distributions (see Proposition \ref{prop:anti-Cexp}).


  \item We prove that for any $u\in \mathscr{D}'(\R^n)$, if $u$ has compact support, then there is a suitable finite iteration of anti-derivatives of $u$ which belongs to $C^0(\R^n)$, where $C^0(\R^n)$ stands for the continuous functions on $\R^n$.
  \item We prove the equality
	  $$C^0(\R^n)\cap \mathscr{D}'_{\cCexp}(\R^n) = C^0_{\cCexp}(\R^n) := C^0(\R^n)\cap \cCexp(\R^n),$$
where $\mathscr{D}'_{\cCexp}(\R^n)$ stands for the $\cCexp$-class distributions on $\R^n$ and $\cCexp(\R^n)$ for the $\cCexp$-class functions on $\R^n$  (see Definition \ref{def:distrCexp} and Proposition \ref{prop:Cexp.and.cont.is.CexpFun}).
	  We do this, using a theorem of \cite{CCMRS} that states that a limit of a $\cCexp$-class family of
	  $\cCexp$-class functions, if it exists, is of $\cCexp$-class (see Proposition \ref{prop:limits}).
  \item It is now enough to prove \blue{Theorem \ref{mainthm}} for $\cCexp$-class locally $L^1$-functions, which is done in Section \ref{sec:holon}.

  \item This case of $\cCexp$-class locally $L^1$-functions is slightly intricate and is based on a fine embedded resolution, or rather, alteration result for subanalytic functions (see Theorem \ref{thm:res:def}). Indeed, by the presence of the new transcendental $\gamma$ functions, $\cCexp$-class functions in fact come from functions on a higher dimensional space, and a large part of the work serves to reduce to an $L^1$ situation on this higher dimensional space, all the while controlling holonomicity (see e.g. Propositions \ref{prop:Cexp:int} \blue{and} \ref{lem:u:int}).

\end{enumerate}


In Section \ref{sec:gen} we indicate a further generalization, namely when enriching the class of $\cCexp$-functions by furthermore allowing complex powers. 

Let us give an example to show that the real-analytic notion of holonomicity is not automatically preserved under Fourier transform of tempered distributions, whereas tempered $\cCexp$-class distributions are stable under Fourier transform and holonomic at the same time by Proposition \ref{prop:Four} and Theorem \ref{mainthm}, which addresses a question raised in \cite{AizC}.
\begin{example}\label{ex}
For $a \in \R$, let $\delta_a\in  \mathscr{D}'(\R)$ be the Dirac delta distribution at the point $a$. Let $\phi:\N\to \Q\cap [0,1]$ be a bijection.
Let $u=\sum_{i=1}^\infty \delta_{\phi(i)}2^{-i}$. Let $\eta$ be the Fourier transform of $u$. It is easy to see that $\eta$ can be extended to an entire function. In particular, $\eta$ is real analytic and hence analytically holonomic. However, the Fourier transform of $\eta$ is not analytically holonomic.
\end{example}

We start in Section \ref{sec:Cexp} by recalling some of the theory of $\cCexp$-class functions from \cite{CCMRS}. In Section 3 we recall and adapt some theory of distributions and their holonomicity, mainly based on \cite{Bjork}. In Section \ref{sec:uCexp} we introduce and study distributions of $\cCexp$-class. In Section \ref{sec:antider} we study anti-derivatives of distributions of $\cCexp$-class and prove Theorem \ref{thm:derivative} and Corollary \ref{thm:gen.sm}. After treating our resolution results in Section \ref{sec:resol}, we are able to finalize the proof of Theorem \ref{mainthm} which is done in Section \ref{sec:holon}.







\section{Functions of $\cC$-class and of $\cCexp$-class}\label{sec:Cexp}

Let us recall some key notions and properties of \cite{CCMRS}, 
in particular, the ring of functions $\cCexp(X)$ for any subanalytic set $X\subset \RR^n$. We will use the language and the results of o-minimality rather freely, the reader is referred to \cite{vdD}, and, likely, we will use the theory of subanalytic sets and functions rather freely, see e.g.~\cite{LR} \cite{Paru2}, \cite{MillerD}.

Call a set $X\subset \RR^n$ for $n\ge 0$ subanalytic if it is definable in the o-minimal structure $\RR_{\rm an}$, the expansion of
the ordered real field by all restricted analytic functions. This notion corresponds to Hironaka's notion of `globally subanalytic' subsets of $\RR^n$ (see e.g.~Definition 1.1 of \cite{CCMRS} and the text below it).
Call a function $f:X\to Y$ subanalytic if ($X$ and) $Y$ and the graph of $f$ are subanalytic sets. Write $\R_{>0}$ for the set of positive real numbers, and $\R_{\ge 0}$ for the set of nonnegative real numbers.

\begin{defn}
\label{def: constructible} For  $X$ a subanalytic set, 
define $\cC(X)$ as the ring of real-valued functions on $X$ generated by all subanalytic functions $X\to\RR$ and by all the functions of the form $\log (g)$ for subanalytic functions $g:X\to\RR_{>0}$.
Functions in  $\cC(X)$ are called constructible functions, or $\cC$-class functions. 
\end{defn}

\begin{defn}
\label{def: gamma function} Consider an integer $\ell\ge 0$, a subanalytic set
$X$ and a subanalytic function $h: X\times\RR\to\RR$. Suppose that for all $x \in X$, the map $t\mapsto h(x,t)$ lies in $L^{1}(\RR)$. For such data
define the function $\gamma_{h,\ell}\colon X\to\CC$ as the oscillatory integral
\[
	\gamma_{h,\ell}(x)=\int_{\RR}h(x,t) (\log|t|)^{\ell}\mathrm{e}^{\mathrm{i}t}\, \mathrm{d}t,
\]
where $\mathrm{e}^{\mathrm{i}t}=\exp(\mathrm{i}t)$ with $\mathrm{i}^2=-1$; note that this integral is automatically finite 
(see the explanation in Definition~2.5 of  \cite{CCMRS}).
\end{defn}

\begin{defn}
\label{def:Cexp} For $X$ a subanalytic set,
define $\Cexp(X)$ as the ring of complex valued
functions on $X$ generated by all functions in $\cC(X)$, all functions of the form $x\mapsto \exp(\mathrm{i}f(x))$ for subanalytic $f:X\to \RR$, and all functions of the form $\gamma_{h,\ell}$
for $\ell\ge 0$ and subanalytic $h:X\times\RR\to \RR$ satisfying for each $x$ in $X$ that  $t\mapsto h(x,t)$ is in $L^{1}(\RR)$.
A function in $\Cexp(X)$ 
is called a function of $\Cexp$-class.
\end{defn}

Note that the definition for $\Cexp(X)$ given in Definition \ref{def:Cexp} is equivalent to the one of \cite[Definition 2.7]{CCMRS} by \cite[Remark 2.6 and Corollary 2.13]{CCMRS}.
The motivating property for the $\cCexp$-class functions is their stability under integration:

\begin{theorem}[Theorem 2.12 \cite{CCMRS}]\label{thm:stab}
Let $X$ be a subanalytic set, $m\ge 0$ be an integer, and let $f$ be in $\cCexp(X\times \RR^m)$. Then there exists a function $g$ in $\cCexp(X)$ such that
$$
g(x) = \int_{y\in \RR^m} f(x,y) \mathrm{d}y
$$
holds for each $x$ in $X$ satisfying that the integral on the right hand side is finite.
\end{theorem}

\begin{remark}\label{rem:Cexp}
By Remark 2.14 of \cite{CCMRS}, the $\cCexp$-class functions forms the smallest class of functions which contains the subanalytic functions and which is stable under taking parametric integrals, parametric Fourier transforms, and precompositions with subanalytic functions (as in Lemma \ref{prop:compCexp}).
\end{remark}


%
%


The following auxiliary lemma gives an alternative description for the functions in $\cCexp(X)$.

\begin{lem}\label{prop:h1h2}
\label{module}
Let $X$ be a subanalytic set and let $f$ be in $\cCexp(X)$. Then there 
is a function $H$ in $\cC(X\times \RR)$ 
such that, for each $x$ in $X$, 
the function $t\mapsto H(x,t)$ is $L^1$  and
$$
f(x) = 
\int_{t\in \RR} H(x,t) \exp(\mathrm{i}t) \mathrm{d}t. 
$$
\end{lem}
\begin{proof}
By Theorem \ref{thm:stab} and \cite[Remark 2.6, Definition 2.7 and Corollary 2.13]{CCMRS}, it is enough to show that the product function
$$
x\mapsto \exp(\mathrm{i}g(x)) \cdot \gamma_{h,\ell}(x)
$$
on $X$ is of the form mentioned in the lemma, with $g:X\to\RR$ subanalytic, $\ell\ge 0$ an integer, and $h:X\to\RR$ a subanalytic function satisfying for all $x\in X$ that $t\mapsto h(x,t)$ lies in $L^{1}(\RR)$.
Observe the following
$$
f(x) := e ^{\mathrm{i}g(x)} \gamma_{h,\ell}(x) = e ^{\mathrm{i}g(x)} \int_{t\in\RR} h(x,t)(\log |t|)^\ell e^{\mathrm{i}t } \mathrm{d}t
= \int_{t\in\RR} h(x,t)(\log |t|)^\ell e^{\mathrm{i} ( t +g(x) )} \mathrm{d}t.
$$
Now let $h'$ be the function sending $(x,t)$ to $h(x,t-g(x))$. Clearly $h'$ is subanalytic, as a composition of subanalytic functions. Then we can write, by the change of variables formula,
$$
f(x) = \int_{t\in\RR} h'(x,t)(\log |t-g(x)|)^\ell e^{\mathrm{i} t} \mathrm{d}t
$$
and we are done by putting $H(x,t):=h'(x,t)(\log |t-g(x)|)^\ell$, which clearly lies in $\cC(X\times \RR)$, and which satisfies clearly the condition of integrability in the $t$-variable (see a short explanation in Definition~2.5 of  \cite{CCMRS}).
\end{proof}



Another key result about the $\cCexp$-class says that pointwise limits can be taken within the $\cCexp$-class, as follows.

\begin{prop}[Proposition 7.7 of \cite{CCMRS}]\label{prop:limits}
	Let $f$ be in $\cCexp(X\times \R)$ for some subanalytic set $X \subset \R^n$. There \blue{exists} $g$ in $\cCexp(X)$
such that, for those $x$ in $X$ such that the limit $\lim_{y\to+\infty}f(x,y)$ exists in $\R$, one has
$$
\lim_{y\to+\infty}f(x,y) = g(x).
$$
Furthermore, there is a $\cCexp$-class function $h$ on $X$ such that the set
$$
\{x\in X\mid \mbox{$\lim_{y\to+\infty}f(x,y)$ exists in $\R$}\}
$$
equals the zero locus of $h$ in $X$.
\end{prop}

\begin{prop}\label{prop:loci}
Let $f$ be in $\cCexp(U)$ for some subanalytic open set $U \subset \R^n$. Then, there is a subanalytic closed set $C\subset U$ with $\dim C < n$ such that $f$ is analytic on $U\smallsetminus  C$.
\end{prop}
\begin{proof}
Write $f$ as an integral as in Lemma \ref{module}, so that
$$
f(x) = 
\int_{t\in \RR} H(x,t) \exp(\mathrm{i}t) \mathrm{d}t 
$$
for some $H$ in $\cC(U\times \RR)$ such that, for each $x$ in $U$,
the function $t\mapsto H(x,t)$ is $L^1$. Write $H$ as a finite sum of finite products of subanalytic functions $h$ on $U\times \RR$ and functions of the form $\log g$ for some positively valued subanalytic functions $g$ on $U\times \R$. Take a cylindrical, analytic cell decomposition so that on each occurring cell, the restrictions of all the occurring $g$ and $h$ are analytic (by the word cylindrical is meant that the projection to the first $i$ coordinates is again a cell decomposition, for each $i$). Note that the restriction of $H$ to any occurring open cell is analytic. Take all the open occurring cells $A$ in $U\times \R$, and let $C$ be the complement of the union of the coordinate projections of these open cells $A$ to $U$. Then $C$ is as desired. Indeed, an integral $\int_{a(x)<t<b(x)}F(x,t) dt$ of an analytic function $F$ on an analytic cell $A$ with cell walls $a(x)$ and $b(x)$ bounding $t$ is analytic in $x$ running over the projection of $A$ to $\RR^n$. For the notion of walls of cells, see Definition 4.2.5 of \cite{CPW}.
\end{proof}


%
%
%



For an open set $X\subset \RR^n$, write $C_{c}^{\infty}(X)$ for the 
$C^\infty$ functions with compact support on $X$.
	\begin{prop}\label{prop:cutoffCexp} If $X$ is a subanalytic open set in $\RR^n$ and $K\subset X$ is a compact subanalytic subset then one can find $\chi \in C_{c}^{\infty}(X)$ of $\cCexp$-class with $0\leq \chi \leq 1$ so that $\chi=1$ in a neighborhood $K'$ of $K$. Furthermore, if $K$ is a closed ball, then we can ensure that the support of $\chi$ is subanalytic as well.
	\end{prop}
	\begin{proof}
		We adapt the proofs of Lemma 1.2.3 and Theorem 1.4.1 of \cite{Hormander83} to our needs. Consider the following function on $\R$
		$$f:t\mapsto \left\{
			\begin{array}{ll}
				e^{-\frac{1}{t}} & \text{if $t>0$} \\
				0 & \text{if $t\leq 0$}
			\end{array}
		\right. .
		$$
		This function is known to be in $C^{\infty}(\RR)$ and also of $\cCexp$-class by (see Example 7.4 of \cite{CCMRS}).
		Next we consider the function $\rho$ on $\R^n$ defined by
		$$ \rho : x \mapsto \frac{1}{\int_{\RR^n} f(1-||t||^2) dt} f(1-||x||^2),$$ with $||x||$ the Euclidean norm on $\R^n$.
		This function is of $\cCexp$-class, lies in $C_{c}^{\infty}(\RR^n)$ with support the unit ball $B$, and satisfies $\int \rho  dx=1$.
		Now let $\eps>0$ be sufficiently small so that
		$$\forall x\in K,\: \forall y \in \RR^n \smallsetminus  X,\:|x-y|\geq 4\eps$$ and
		let $v$ be the characteristic function of $K_{2\eps}$ where we use the notation
		$$K_{\eps} := \{y \in \RR^n \mid \exists x \in K,\: |x-y|\leq \eps\}$$
and likewise for $K_{2\eps}$.
		As $K$ is a subanalytic set, $v$ is subanalytic and in particular of $\cCexp$-class. Consider
		$\rho_{\eps} : x \mapsto \frac{1}{\eps^{n}}\rho\left( \frac{x}{\eps} \right)$ which has as support the ball
		$\{x \in \RR^n \mid  ||x||\leq\eps\}$ and satisfies $\int \rho_{\eps} dx=1$. So, the convolution
		$$\chi := v*\rho_{\eps}$$
		is a $C^\infty$ function with compact support in $K_{3\eps}$ and by Theorem \ref{thm:stab} it is also of $\cCexp$-class, and one easily checks that
		$1-\chi=(1-v)*\rho_\eps$ vanishes on $K_\eps$ which is a neighborhood of $K$. Hence, $\chi$ is as desired.
Let us now assume that $K$ is a closed ball $\Bf(a,R)$ around a point $a$ of radius $R>0$. Note the following in this case.
			\begin{itemize}
				\item One has $K_{2\eps}=\Bf(a,R+2\eps)$.  \\
					( ``$\subset$'' is obvious, and for ``$\supset$'' :
				for $y \in \Bf(a,R+2\eps)\smallsetminus  \{a\}$, take $x=a+\frac{y-a}{\norm{y-a}}R \in \Bf(a,R)$, we have $\norm{x-y}\leq 2\eps$.)
			        \item We know that $\Supp(\chi)$, the support of $\chi$, is a closed subset of
			$$\Supp(v)+\Supp(\rho_{\eps}) = K_{2\eps}+K_{\eps} = K_{3\eps}.$$
			For the last equality : ``$\subset$'' is obvious and for ``$\supset$'', if $z \in K_{3\eps}\smallsetminus  K_{2\eps}$, then we can write $z=x+y$ with $y=a+\frac{z-a}{\norm{z-a}}(R+2\eps) \in K_{2\eps}$ and $x=z-y \in K_\eps$. If $z \in K_{2\eps} \subset  K_{3\eps}$, we can write $z=x+y$ with $x=0 \in K_\eps$ and $z=y \in K_{2\eps}$.
		\item  Let $z$ be in the open ball $B(a,R+3\eps)$ around $a$ of radius $R+3\eps$. Define $y=a+\frac{z-a}{\norm{z-a}}(R+2\eps)$.
			One has  $y \in K_{2\eps}\cap B(z,\eps)$, and there is $r>0$ such that $B(y,r)\subset B(z,\eps)$ and such that the measure of the $B(y,r)\cap K_{2\eps}$ is nonzero. Thus, one has $\chi(z)>0$ and $z \in \Supp(\chi)$. Thus, $B(a,R+3\eps)\subset \Supp(\chi)$, and as the support is closed, we conclude that $K_{3\eps}=\Supp(\chi)$.
	                \end{itemize}
		This shows that when $K$ is a closed ball, $\chi$ can be taken such that its support is subanalytic as well.
	\end{proof}


The following lemma about precomposition with subanalytic functions easily follows from the definition of $\cCexp$-class functions.
	\begin{lem}\label{prop:compCexp} Consider a subanalytic function $g:X\to Y$ for some subanalytic sets $X$ and $Y$ and let $f$ be in $\cCexp(Y)$. Then the composition $f\circ g$ lies in $\cCexp(X)$.
	\end{lem}
	\begin{proof}
Immediate from the definitions.
\end{proof}

We prove an auxiliary Proposition \ref{prop:Cexp:int}, based on the following lemma and its corollary.

\begin{lem}\label{lem:H:int}
Suppose that $H:\R^n\to \R$ is of $\cC$-class. 
Then there exists a bounded subanalytic $g:\R^n\to \R_{>0}$ 
such that the product $gH$ is integrable over $\R^n$.
\end{lem}
\begin{proof}
Let us first treat the case that $H$ is subanalytic.
Let $H_0$ be the function sending $x$ in $\R^n$ to $1/|H(x)|$  whenever $|H(x)|>1$ and to $1$ elsewhere. Let $H_1$ be the function sending $x$ to $1/|x|^{2n}$ whenever $|x|>1$ and to $1$ elsewhere.
Let $g$ be the product $H_0 H_1$. Then $g$ is clearly as desired.

Now let $H$ be of $\cC$-class. Write $H$ as $\sum_i (\prod_j \log g_{ij}(x)) f_i(x)$, for finitely many subanalytic functions $g_{ij}>0$ and $f_i$. Apply the previous construction to the subanalytic function $\sum_i ((1+|f_i|) \cdot \prod_j(1 + g_{ij}+1/g_{ij}))$, to construct $g$. Then one can check that $g$ is as desired, as $H(x)g(x)\le \min(c,c/|x|^{2n})$ for each $x$ and some constant $c$, by construction. \end{proof}

\begin{cor}\label{cor:int:cC}
Let $H:\R^{n+1}\to\R:(x,t)\mapsto H(x,t)$ be of $\cC$-class. Suppose that for each $x$ in $\R^n$, the function $t\mapsto H(x,t)$ is integrable over $\RR$. Then there exists a bounded subanalytic function $g: \R^n\to\R_{>0}$  such that $(x,t)\mapsto g(x)H(x,t)$ is integrable over $\R^{n+1}$.
\end{cor}
\begin{proof}
By Corollary 3.5 of \cite{CMillerIMRN}, we may reduce to the case that $H$ equals zero outside a cell
\blue{$A \subset \RR^{n}\times (0,1)$ or  $A \subset \RR^{n}\times (1,+\infty)$} and that on $A$, with $x$ in $\R^n$ and $t$ in $\R_{>0}$, one has a finite sum
$$
H(x,t) = \sum_{i}d_i(x) S_i(x,t) t^{\alpha_i}\blue{(\log (t))^{\ell_i}}
$$
for \blue{some integers $\ell_i\ge 0$,} some functions $d_i$ of $\cC$-class and depending only on $x$, some positive bounded subanalytic functions $S_i$, and such that for each $x$ and each $i$, the function $t\mapsto t^{\alpha_i}\blue{(\log(t))^{\ell_i}}$ is integrable over $A_x:=\{t\in\R\mid (x,t)\in A\}$. \blue{Note that the function $t\mapsto t^{\alpha_i}\blue{(\log(t))^{\ell_i}}$ has constant sign}. 
Clearly (up to partitioning subanalytically the space of the $x$-variables \blue{and slightly rewriting the above finite sum}) we may suppose for each $i$ that $d_i$ is either non-negative, or, non-positive on the projection of $A$ to $\R^n$. 
Consider $H_i(x) := \int_t d_i(x)S_i(x,t)|t|^{\alpha_i}\blue{(\log(t))^{\ell_i}}dt$. By Lemma \ref{lem:H:int} for $H_i$ for each $i$, we find a subanalytic $g_i(x)$ such that $H_i(x)g_i(x)$ is $L^1$ and $g_i(x)$ is bounded and everywhere positive. 
Now put $g ( x ) := \prod_i g_i(x).$
Then clearly
$$
gH : \R^{n+1}\to\R : (x,t)\mapsto g(x)H(x,t)
$$
  is $L^1$ by Fubini-Tonelli, as desired.
\end{proof}
\begin{remark}\label{rem:3.3:IMRN}
Note that in Definition 3.3 of \cite{CMillerIMRN}, the power series $F$ has to converge on an open neighborhood of the closure of the image of $\varphi$. This very word `closure' should be added in Definition 3.3 of \cite{CMillerIMRN}, in line with previous preparation results as in \cite{LR} \cite{Paru2}, \cite{MillerD}; this correction amends Definition 3.3 of \cite{CMillerIMRN}.
\end{remark}


\begin{prop}\label{prop:Cexp:int}
Let $f:\R^n\to\C$ be of $\cCexp$-class. Then there exists a bounded subanalytic function $g:\R^n\to\R_{>0}$ such that the product $gf$ is integrable over $\R^n$. Even more, one can write $f(x)$ as $\int_{t\in \RR} H(x,t) \exp(\mathrm{i}t) dt$ as in Lemma \ref{module} such that moreover $(x,t)\mapsto g(x)H(x,t)$ is integrable over $\R^{n+1}$.
\end{prop}
\begin{proof}
Write $f$ as $\int_{t\in \RR} H(x,t) \exp(\mathrm{i}t) dt$  as given by Lemma \ref{module}. Note that the function $H$ is of $\cC$-class on $\RR^{n+1}$ and that for each $x\in\R^n$, the map $t\mapsto H(x,t)$ is $L^1$.
Now apply Corollary \ref{cor:int:cC} to the function $H$ on $\R^{n+1}$ to find $g$ as desired.
\end{proof}


%
%
%
%
%

\section{D-modules and distributions}
Our main reference for the relation between D-modules and distributions in the analytic context is \cite{Bjork}, but we will use slightly different notation. We will now fix notations and briefly recall the results and definitions that we need.

\subsection{Basic notation}

We recall some rather standard notation.

\begin{notation}\label{not:left.right}  Consider a nonempty open set $U\subset \RR^n$.
	\begin{itemize}
		\item $C^0(U)$ stands for the algebra of continuous (complex valued) functions on $U$. For $r>0$ write $C^r(U)$ for the $r$ times continuously differentiable functions on $U$.

		\item $C^r_c(U)$ stands for the algebra of compactly supported functions in $C^r(U)$, for $r\ge 0$. Note that this means that the support (taken inside $U$) has to be compact.

		\item $\mathscr{D}'(U)$ stands for the complex vector space of 
			distributions on  $U$.

		\item $\mathscr{D}_c'(U)$ stands for the complex vector space of compactly supported distributions on $U$.


		\item $\cS(U)$ stands for the complex vector space of Schwartz functions on  $U$.

		\item $\cS'(U)$ stands for the complex vector space of tempered distributions on  $U$.






	\end{itemize}
\end{notation}

\subsection{Holonomic distributions}\label{subsec:holonomic}
\begin{notation} \label{not:hol}~
		\begin{itemize}
			\item For any integer $n>0$, denote by $\scD'_{\RR^n}$ the sheaf of distributions on $\RR^n$. In particular, for any nonempty open set $U \subset \RR^n$, one has  $\scD'_{\RR^n}(U)=\scD'(U)$.
			\item Denote by $\cO_{\CC^n}$ the sheaf of analytic functions on $\CC^n$.
			\item Denote by $\cD_{\CC^n}$ the sheaf of rings of differential operators on $\CC^n$ with analytic coefficients. For a nonempty open set $\Omega$ in $\CC^n$, write $\cD_{\CC^n \mid \Omega}$ for the restriction of the sheaf $\cD_{\CC^n}$ to $\Omega$.
%
			\item We consider $\scD'_{\RR^n}$ as a left $\cD_{\CC^n}$-module on $\CC^n$, as follows:
				\begin{itemize}
					\item For any open set $\Omega$ of $\CC^n$, $U:=\Omega \cap \RR^n$ is an open set of $\RR^n$ and we define $\scD'_{\RR^n}(\Omega)$ as $\scD'_{\RR^n}(U)$.
					\item The action ``.'' is defined from the following base cases: for any $u \in \scD'_{\RR^n}(U)$,
						for any derivation $\partial$, for any analytic function $a$ on $\Omega$, we have by definition
						$$\partial.u=\partial u\:\:\text{and}\:\:a.u=au,$$
						namely, for any $\varphi \in C_c^{\infty}(U)$,
						$$\la \partial.u,\varphi \ra = - \la u,\partial \varphi \ra \:\:\text{and}\:\:
						\la a.u,\varphi \ra = \la u,a\varphi \ra.$$
				\end{itemize}
		\end{itemize}
\end{notation}

The following should not be confused with the more restrictive notion of algebraic holonomicity.

\begin{definition}[cf. {\cite[page 298]{Bjork}}, {\cite{Andronikof}}, {\cite{Kashiwara-reg-hol-dist}}]\label{def:holonomic}
	For any open set $U \subset \RR^n$, say that a distribution $u \in \scD'(U)$ is holonomic if for each $x \in U$ there is an open neighborhood $\Omega$ of $x$ in $\CC^n$ and a coherent left ideal sheaf $\cI$ of  $\cD_{\CC^n \mid \Omega}$ such that $\cI.u=0$ on $\Omega$ and the quotient sheaf $$\cD_{\CC^n \mid \Omega}/\cI$$ is an analytic holonomic $\cD_{\CC^n \mid \Omega}$-module. (Alternatively, one could require that the dimension of the characteristic variety $\Char\!\left(\cD_{\CC^n \mid \Omega}/\cI\right)$ is $n$, see Remark \ref{rem:hol}.) %
\end{definition}

\begin{remark}\label{rem:hol}
	Recall that for any differential operator of order $m$
	$$P(x,\partial)=\sum_{\abs{\alpha}\leq m}a_{\alpha}(x)\partial^{\alpha}$$
and with $a_{\alpha}\not=0$ for at least one index tuple $\alpha$ with $\abs{\alpha}=m$,
	one defines its principal symbol as the homogeneous polynomial
	$$p_{m}(x,\xi)=\sum_{\abs{\alpha}=m}a_{\alpha}(x)\xi^{\alpha} \in \cO_{\CC^{n}}[\xi_1,\dots,\xi_n].$$
For a coherent left ideal sheaf $\cI$ of  $\cD_{\CC^n \mid \Omega}$ with an open set $\Omega\subset \CC^n$, the characteristic variety $\Char\!\left(\cD_{\CC^n \mid \Omega}/\cI\right) \subset \CC^{2n}$ is the locus of the common zeros of the principal symbols of the elements of $\cI(\Omega)$.
	For more details see for instance \cite[\S I.6, Definition 3.1.1]{Bjork} or \cite{Granger-Maisonobe}.
\end{remark}


%
\begin{prop}\label{prop:sum:holon}
Let $U \subset \RR^n$ be an open set. Let $V \subset \CC^n$ be an open set such that $V\cap \RR^n = U$.
The set of holonomic distributions on $U$ is a $\cD_{\CC^n}(V)$-module.\\
In particular :
\begin{itemize}
	\item If $u$ and $v$ are holonomic distributions, then so is $u+v$.
	\item For any differential operator $P \in \cD_{\CC^n}(V)$ and for any holonomic distribution $u$ on $U$, the distribution $P.u$ on $U$ is holonomic.
\end{itemize}
\end{prop}
\begin{proof}
	The proof follows the same ideas as for \cite[Proposition 7.4.7]{Bjork}. Assume that $u$ and $v$ are holonomic distributions on $U$. Then, by definition, for any $x \in U$, there is an open neighborhood $\Omega$ of $x$ in $\CC^n$ (that we can take contained in $V$), there are two coherent left ideal sheafs $\cI_u$ and $\cI_v$ of $\cD_{\CC^n \mid \Omega}$ with $\cI_u.u=0$ and $\cI_v.v=0$ on $\Omega$ and such that $\cD_{\CC^n \mid \Omega}/\cI_u$ and $\cD_{\CC^n \mid \Omega}/\cI_v$ are holonomic $\cD_{\CC^n \mid \Omega}$-modules.
	\begin{itemize}
		\item   Consider the coherent left ideal sheaf $\cI_u \cap \cI_v$. Then we have $(\cI_u \cap \cI_v).(u+v)=0$.
			Furthermore $\cD_{\CC^n \mid \Omega}/\cI_u \cap \cI_v$ is a submodule of
			$\cD_{\CC^n \mid \Omega}/\cI_u \oplus \cD_{\CC^n \mid \Omega}/\cI_v$ and hence is holonomic.
			Thus, $u+v$ is a holonomic distribution on $U$.
		\item 	Let $P$ be a partial differential operator. We consider the coherent left ideal sheaf
			\[(\cI_{u}:P)=\{Q \in \cD_{\CC^n \mid \Omega} \mid QP \in \cI_{u}\}.\]
			By construction we have $(\cI_{u}:P).(P.u)=0$.\\
			Furthermore $(\cI_{u}:P)$ is the kernel of the morphism
			\[\begin{array}{ccl}
					\cD_{\CC^n} & \to & \cD_{\CC^n}/\cI_{u} \\
					Q & \mapsto & QP \mod \cI_{u}
				\end{array}
			\]
			and thus $\cD_{\CC^n \mid \Omega}/(\cI_{u}:P)$ is isomorphic to a $\cD_{\CC^n \mid \Omega}$-submodule of $\cD_{\CC^n \mid \Omega}/\cI_{u}$ so it is again holonomic (for instance by \cite[Proposition 3.1.2]{HTT} or \cite[Remark 3.1.2]{Bjork}). It follows that $P.u$ is a holonomic distribution on $U$.
	\end{itemize}

\end{proof}


We will reduce the proof of Theorem \ref{mainthm} to the base case of Proposition \ref{cor:base-case-u} via a rather involved reduction argument. We first develop some basic results that will be useful for this reduction. Proposition \ref{cor:base-case-u} itself will be proved in the next section.



\begin{proposition}[Multiplication with an analytic unit]\label{prop:product:unit}
	Let $u$ be a distribution on an open set $U\subset \RR^n$. Let $v:U\to\RR$ be an analytic unit on the support of $u$ (namely, an analytic function whose restriction to an open neighborhood in $U$ of the support of $u$ is nonvanishing). Then
$u$ is holonomic if and only if $v.u$ is holonomic.
\end{proposition}
\begin{proof}
This follows from the definition of the characteristic variety and Proposition~\ref{prop:sum:holon}.
\end{proof}

\begin{proposition}[Multiplication with a coordinate function]\label{prop:mult}
	Let $u$ be a distribution on $\RR^n$ which is given by a locally $L^1$ function. Then,
$u$ is holonomic if and only if $x_1u$ is holonomic.
\end{proposition}

\begin{proof}	
	Assume first that $u$ is holonomic. Then, by Proposition \ref{prop:sum:holon}, $x_1u$ is also holonomic.\\
	Assume now that $x_1u$ is holonomic. Let $f=x_1^2u$. It follows from the previous case that, $f$ is a holonomic distribution.
	We consider
	$$\Omega=\{\lambda \in \CC\mid \Re(\lambda)>-1\}\:\text{and}\:\bar{\Omega}=\{\lambda \in \CC\mid \Re(\lambda)\geq-1\},$$
where $\Re$ takes the real part of a complex number.  	

	For any $\lambda\in \bar{\Omega}$, the function $f_\lambda:=|x_1^2|^\lambda  f $ is a locally $L^1$ function and induces a distribution also denoted by $f_\lambda$. Remark that $f_{-1}=u$. For any test function $\varphi \in C_{c}^{\infty}(\RR^n)$, the application
	$\lambda \mapsto \la f_{\lambda},\varphi \ra$
	is continuous on $\bar{\Omega}$ and analytic on $\Omega$. We consider the assignment $F:\lambda \in \bar{\Omega} \mapsto f_\lambda$. By \cite[Theorem 7.6.1]{Bjork}\footnote{\cite[Theorem 7.6.1]{Bjork} proves this fact for regular holonomic distributions on complex analytic manifolds. However its proof does not use the regularity assumption and works in the real context in the same way as in the complex one.}, the assignment $F_{\mid \Omega}$ can be meromorphically continued from $\Omega$ to $\C$. By the uniqueness theorem this continuation coincides with $F$ on $\bar{\Omega}$. By \cite[Theorem 7.6.1]{Bjork} all the coefficients of the Laurent series of this meromorphic continuation in any $\lambda$ are holonomic. In particular $F$ is defined in $-1$ and its zero coefficient in its Laurent series expansion is $F(-1)=u$, implying that $u$ is holonomic. 
\end{proof}

\begin{proposition}[Multiplication with half space]\label{prop:mult:half}
Let $u$ be a distribution on $\RR^n$ which is given by a locally $L^1$ function. Let $g$ be the characteristic function of $\R_{\ge 0}\times \R^{n-1}$. If
$u$ is holonomic then $gu$ is holonomic.
\end{proposition}
\begin{proof}
The proof is similar to the proof of the previous proposition. We will continue to use the notation $\Omega$.
		Let $h:\R\to \R$ be given by $h(x)=x$ for positive $x$ and $0$ otherwise.
		For $\lambda\in \Omega$,  Let $g_\lambda:=h(x_1)^\lambda u$.
		As in \cite[Theorem 7.6.1]{Bjork}\footnote{\cite[Theorem 7.6.1]{Bjork} gives the result for  $h$ replaced with $|\cdot|$; however, the same argument is still valid for our case.} the distribution $g_\lambda$  is holonomic for any $\lambda\in  \Omega$, and in particular for $\lambda = 0$ as required. 
\end{proof}

\begin{theorem}[Push-forward]\label{thm:push}\label{prop:push:holon}
	Let $\phi:U\subset \R^m\to V\subset \R^n$ be an analytic map for some open subsets $U\subset \R^m$ and $V\subset \R^n$. Let $u$ be a holonomic distribution on $U$. Assume that $\phi$ is proper on the support of $u$. Then the distribution $\phi_*(u)$ on $V$  is holonomic.
\end{theorem}
\begin{proof} In \cite[Theorem 7.4.11, \S 7.4.12]{Bjork}, Bj\"{o}rk proves this fact for regular holonomic distributions on complex analytic manifolds. However its proof does not use the regularity assumption and works in the real context in the same way as in the complex one.
\end{proof}

\begin{proposition}[Analytic transformations]\label{prop:an}
	Let $u$ be a distribution on an open set $V\subset \RR^n$. Let $f: V \to\RR^n$ be a map such that the restriction of $f$ to an open set $U$ containing the support of $u$ is proper and bi-analytic. Then,
$u$ is holonomic if and only if the push-forward $f_*(u)$ is holonomic.
\end{proposition}
\begin{proof}
	By assumption $f:U \to U':= f(U)$ is a proper analytic transformation where $U$ is an open set of $\RR^n$ containing the support of $u$. The set $U'$ is an open set containing the support of  $f_*(u)$, $f^{-1}:U'\to U$ is proper and bi-analytic on $U'$ and we have $u=(f^{-1})_*(f_*(u))$, so the equivalence follows from Theorem \ref{thm:push} applied to $u$ and $f_*u$ with $f$ and $f^{-1}$.  
\end{proof}

\begin{proposition}[Cartesian product situation]\label{prop:product-situation}
Let $u_j$ be a distribution on an open set $U_j\subset \RR^{n_j}$ for $j=1,2$. Suppose that each $u_j$ is holonomic. Then the distribution $u_1\otimes u_2$ on $U_1\times U_2$ is also holonomic.
Here, $u_1\otimes u_2$ is defined as
\begin{equation*}\label{def:prodtensdist}
	\la u_1\otimes u_2, (x_1,x_2)\mapsto \varphi(x_1,x_2) \ra  =  \la u_1,x_1 \mapsto \la u_2,x_2 \mapsto \varphi(x_1,x_2)\ra \ra.
\end{equation*}
for any $\varphi \in C_{c}^{\infty}(U_1\times U_2)$.
\end{proposition}
\begin{proof}
For any $j \in \{1,2\}$, the distribution $u_j$ is holonomic on $U_j$, and thus, for any $x_j \in U_j$, there is an open neighborhood $\Omega_j$ of $x_j$ in $\CC^{n_j}$ and a coherent left ideal sheaf $\cI_j$ of  $\cD_{\CC^{n_j} \mid \Omega_j}$ such that $\cI_j.u_j=0$ on $\Omega_j$ and $\cD_{\CC^{n_j} \mid \Omega_j}/\cI_j$ is an analytic holonomic $\cD_{\CC^{n_j} \mid \Omega_j}$-module.\\
		We denote by $\cI_1+\cI_2$ the coherent left ideal sheaf of $\cD_{\CC^{n_1+n_2} \mid \Omega_1 \times \Omega_2}$ generated by $\cI_1$ and $\cI_2$. It follows from the definition of the tensor product (and the independence between variables of $\CC^{n_1}$ and those of $\CC^{n_2}$) that $(\cI_1+\cI_2).(u_1\otimes u_2)=0$.\\
		Furthermore, similarly to \cite[\S 7.5.2]{Bjork}, we have the isomorphism
		$$\cD_{\CC^{n_1+n_2} \mid \Omega_1 \times \Omega_2}/(\cI_1+\cI_2)\simeq
		\cD_{\CC^{n_1} \mid \Omega_1}/\cI_1 \boxtimes
		\cD_{\CC^{n_2} \mid \Omega_2}/\cI_2
		$$
		where $\boxtimes$ denotes the external product of $\cD$-modules (see for instance \cite[\S 2.4.4]{Bjork} or \cite{HTT}), which is holonomic as soon as each factor is holonomic (\cite[Proposition 3.2.2]{HTT}).
It follows that $u_1 \otimes u_2$ is holonomic.
\end{proof}

\begin{prop}
\label{prop:power-map:holon}
Suppose that a distribution $u$ on $\RR^n$ is locally integrable, and that the support of $u$ is included in $(\RR_{\ge 0})^n$.
Let $k>0$ be an integer and consider the map $f:\RR^n\to\RR^n:x\mapsto (x_1^k,x_2,\ldots,x_n)$. Then there exists a distribution $v$ with support of $v$ included in $(\RR_{\ge 0})^n$ and such that $f_*(v)=u$. Furthermore, the distribution $v$ is represented by the locally integrable function which sends $y=(y_1,\dots,y_n)$ in  $(\RR_{\geq 0})^n$ to
	\begin{equation}\label{push:xk} ky_1^{k-1}u(y_1^k,y_2,\dots,y_n)
\end{equation}
 and which is zero elsewhere.
\end{prop}
\begin{proof}
	Let $v$ be the distribution on $\R^n$ given by the locally integrable function which equals (\ref{push:xk}) on $(\RR_{\geq 0})^n$ and zero elsewhere.
Thus, for any test function $\varphi \in C_{c}^{\infty}(\RR^n)$ we have
$$ \la v,\varphi \ra = \int_{(\RR_{\geq 0})^n}ky_1^{k-1}u(y_1^k,y_2,\dots,y_n)\varphi(y_1,\dots,y_n)\mathrm{d}y. $$
Considering the change of variables
$$(y_1,\dots,y_n) \mapsto (y_1^k,y_2,\dots,y_n)=(x_1,\dots,x_n)$$
we see that
$$\la f_{*}v,\varphi \ra = \la v,\varphi\circ f \ra = \int_{(\RR_{\geq 0})^n}u(x)\varphi(x)\mathrm{d}x = \la u,\varphi\ra$$
	meaning that $f_{*}v=u$ as desired.%
This finishes the proof of Proposition \ref{prop:power-map:holon}.
\end{proof}

In our main proof of holonomicity, a reduction to the following base case will be made.

\begin{prop}[Base case]\label{cor:base-case-u} 
The distribution $u$ on $\RR^n$ given by the $L^1$ function $\RR^n\to\CC$ sending $x$ in $(0,1)^n$ to
$$
(\prod_{j=1}^n x_j^{a_j}) (\prod_{j=1}^n (\log x_j)^{b_j}  ) \exp(\mathrm{i} 
\prod_{j=1}^n x_j^{c_j})
$$
and  other $x$ to $0$, with complex $a_j$ with $\Re(a_j)>-1$, integers $b_j\ge 0$, and integers $c_j\le 0$,
is holonomic.


\end{prop}

The key point of Proposition \ref{cor:base-case-u} is the holonomicity around $x$ in $[0,1]^n\smallsetminus  (0,1]^n$. The next section is dedicated to a proof of  Proposition \ref{cor:base-case-u} and makes a detour via algebraic holonomicity.

\section{Algebraic Holonomicity}
In this section we introduce the notion of algebraically holonomic distributions (which is a more restrictive notion than that of holonomic distributions in Definition \ref{def:holonomic}) and prove that certain distributions are algebraically holonomic (and hence holonomic), see Lemma \ref{lem:alg:to:holon}, Corollary \ref{cor:alg:to:holon} and Proposition \ref{prop:base-case-u} below.

We will \blue{rely} on the notion of holonomic algebraic D-modules on smooth algebraic varieties, see for instance \cite{HTT}.

\begin{definition}
Let $\bfX$ be a smooth real algebraic variety.
We denote by $D^a_\bfX$ the sheaf of algebraic differential operators on $\bfX$.

As usual we will refer to a  quasi-coherent \blue{sheaf} of $D^a_\bfX$-modules as  a $D^a_\bfX$-module.

Let $X:=\bfX(\R)$. Let $u\in \mathscr{D}'(X)$ be a distribution. We denote by
$u D^a_\bfX$ the sheaf on $\bfX$ defined
\blue{for any Zariski algebraic open subset $\bfU$ of $\bfX$ by
	$$u D^a_\bfX(\bfU):=u|_{\bfU(\R)} D^a_\bfX(\bfU) := \{ud \mid d \in D^a_\bfX(\bfU)\}$$
	with
	$$\la ud, \varphi \ra := \la u, d\varphi \ra$$
	for any $\varphi \in C_{c}^{\infty}(\bfU(\R))$.
}
This is a sheaf of right $D^a_\bfX$-modules \blue{of distributions}.
\end{definition}

\blue{
	\begin{remark} \label{rem:rightstructure}
		We use here the right structure of $D_{\bfX}^a$-module of distributions. Indeed, working with a left structure needs a choice of a volume form on $X$.
		For an open set of $\RR^n$ both points of view coincide using the Lebesgue measure.
	\end{remark}
}

The following lemma is formulated in the language of Schwartz functions on Nash manifolds. We refer the reader to \cite{AG_Sc} for these notions.

\begin{lem}\label{lem:van}
\label{lem: distributions supported on V(f) are killed by power of f}
    Let $X$ be a Nash manifold and let $Z\subset X$ be a closed (semi-algebraic) subset.
    Then for every tempered distribution $u$ on $X$ which is supported on $Z$ and any 
     Nash function $f$ on $X$ which vanishes on $Z$,
    there exists $N \in \N$ such that $f^N \cdot u=0$.
\end{lem}
\begin{proof}
Without loss of generality, we may assume that $X$ is affine.
Consider a stratification $Z= \bigcup\limits_{i=1}^k S_i$ by Nash submanifolds. We prove the lemma by induction on $k$.

In the case $k=1$ the set $Z$ is a closed Nash manifold.
A distribution $u$ as above is a continuous functional $u : \Sc(X) \to \C$ which vanishes on $\Sc(X\smallsetminus Z)$.
Set
$$W_k(Z) =\{
g\in \Sc(X) | \text{ all partial derivatives of $g$ up to degree $k$ vanish at $x$,}  ~ \forall x \in Z
\}.$$
Then by Borel's lemma (see e.g. \cite[Lemma B.0.9]{AG_SM}) we have an isomorphism of topological vector spaces\footnote{By the Banach open map theorem, the fact that this is an isomorphism of vector spaces implies that this is also an isomorphism of topological vector spaces.}
\[
\Sc(X) / \Sc(X\backslash Z)=\lim\limits_{\leftarrow_k}\left( \Sc(X) / W_k(Z)\right).
\]
\blue{By continuity of $u$, the functional $u$ considered as a functional on the inverse limit, has to be bounded on some open neighborhood of $0$. Such a neighborhood includes the image of $W_N(Z)$ for some N. Thus $u$ must vanish on $W_N(Z)$.}
We conclude that $f^{N+1}\cdot u=0$.


Assume now that the claim holds for any Nash manifold $X'$ and closed semi-algebraic subset $Z' \subseteq X'$ together with a stratification with at most $k-1$ strata.
Order $(S_i)$ such that $S_1$ is closed in $X$. By the induction assumption, the statement is valid when we replace $(X,Z)$ by $(X\smallsetminus S_1, Z\smallsetminus S_1)$. Therefore, there exists a number $N$ such that $$(f|_{X\smallsetminus S_1})^N \cdot u|_{X\smallsetminus S_1}=0.$$
    Let $u':=f^N u$. Then $\langle u', g\rangle=0$ for every $g \in \Sc(X \smallsetminus S_1)$.
    We conclude that $u'$ is supported on $S_1$.
    By the base case of the induction there exists a number $N'$ such that $f^{N'}u'=0$. The required statement now follows as $$f^{N+N'}u=0.$$
\end{proof}

\begin{lemma}\label{lem:coherence}
	\blue{Let $\bfX$ be a smooth real algebraic variety and $X = \bfX(\R)$.}
If $u\in \mathscr{D}'(X)$ is a tempered distribution, then $u D^a_\bfX$ is a coherent $D^a_\bfX$-module.
\end{lemma}
\begin{proof}
Without loss of generality, we may assume that $\bfX$ is affine. The sheaf $u D^a_\bfX$ is locally finitely generated over $D^a_\bfX$ so, by \cite[Proposition 1.4.9]{HTT} it is enough to show that it is $O_{\bfX}$-quasi-coherent. The modules are cyclic, so it is enough to show that the ideal sheaf of annihilators defined by $I(\bfU):=Ann(u|_{\bfU(\R)})$
    is $O_{\bfX}$-quasi-coherent.
    By \cite[Theorem 1.4.1]{EGA1}, this  follows from the fact that if $\bfU=\bfX_f$ is the non-vanishing locus of a regular function $f$ and
    $u|_{\bfU(\R)} d=0$ for some $d\in D^a_\bfX(\bfU)$
    then there exists $n\in \bfN$ such that
    \begin{enumerate}
        \item $df^n\in D^a_\bfX(\bfX)$
        \item $u df^n=0$.
    \end{enumerate}
    This  follows from Lemma \ref{lem:van}.
\end{proof}
\begin{remark}\label{rem:holonomicity}
~
    \begin{itemize}
	    \item Note that item (1) of the above proof does not work for the analytic case. For example, 		
		    \blue{if the distribution $u$ defined by the continuous function $e^g$ with $g$ the continuous function defined by
			    $g(0)=0$ and $g(x)=e^{-1/x^2}$ for $x\neq 0$, and the differential operator $d=\frac{d}{dx}-g'$.
		    }
        \item The sheaf $u D^a_\bfX$ is a coherent sheaf on the  algebraic variety $\bfX $ which is defined over $\R$. It is easy to construct from such a sheaf a coherent sheaf on the complexification $\bfX_\C$. So we can use notions defined for D-modules over algebraically closed fields also in the context of coherent sheaves over $\bfX$.
        \item
		For a smooth function $f\in C^\infty(X)$, we define similarly the left coherent $D^a_\bfX$-module $D^a_\bfX f$ \blue{of smooth functions.}
    \end{itemize}
\end{remark}

\begin{definition}
\blue{Let $\bfX$ be a smooth real algebraic variety and $X = \bfX(\R)$.}
\begin{enumerate}
    \item We say that a tempered distribution $u\in \mathscr{D}'(X)$
    is algebraically holonomic if $u D^a_\bfX$ is an holonomic (right) $D^a_\bfX$-module.
    \item We say that a smooth function $f\in C^\infty(X)$
    is algebraically holonomic if $ D^a_\bfX f$ is an holonomic (left) $D^a_\bfX$-module.
\end{enumerate}
\end{definition}
\begin{lem}\label{lem:hol.eq}
	\blue{Let $\bfX$ be a smooth real algebraic variety and $X = \bfX(\R)$.}
    Let $u\in \mathscr{D}'(X)$ be a tempered distribution. Then the following are equivalent:
    \begin{enumerate}
        \item $u$ is algebraically holonomic,
	\item for each $x \in \bfX(\bR)$ there is a Zariski open neighborhood $\bfU$ of $x$ in $\bfX$ and a coherent right ideal \blue{subsheaf} $\cI$ of  $\blue{D^{a}_{\bfU}}$ such that
		$u|_{\bfU(\bR)}\cI=0$ and the quotient sheaf $\cI \backslash \blue{D^{a}_{\bfU}}$ is an \blue{algebraic} holonomic \blue{right} $\blue{D^{a}_{\bfU}}$-module.
    \end{enumerate}
\end{lem}
\begin{proof}~

    (1)$\Rightarrow$(2):
    Assume $\bfU$ to be an affine neighborhood of $x$ and $\cI$ to be the annihilator of $u|_{\bfU(\bR)}$.
    \blue{  We have the exact sequence
	    \[0 \to \cI \to D_{\bfU}^a \to uD_{\bfU}^a\mid_{\bf U} \to 0,\]
	    the sheaf $D_{\bfU}^a$ is coherent and the sheaf $uD_{\bfU}^a\mid_{\bfU}$ is coherent by Lemma \ref{lem:coherence}. So $\cI$ is coherent because the kernel of any morphism of sheaves modules on a coherent Noetherian scheme (so, in particular, an algebraic variety) is coherent (see for instance \cite[Chapter II, Proposition 5.7]{Hartshorne}. The sheaf $\cI \backslash \blue{D^{a}_{\bfU}}$ is \blue{algebraic} holonomic because it is isomorphic to $uD_{\bfU}^a\mid_{\bfU}$.
	    }

    (2)$\Rightarrow$(1):
    For any $x\in \bfX(\bR)$, let $\bfU_x$ be a Zariski open neighborhood of $x$ in $\bfX$ such that there exists a coherent right ideal sheaf $\cI$ of  $\blue{D^{a}_{\bfU_x}}$ such that $u|_{\bfU(\bR)_x}\cI=0$ and the quotient sheaf $\cI \backslash \blue{D^{a}_{\bfU_x}}$ is an \blue{algebraic} holonomic $\blue{D^{a}_{\bfU_x}}$-module.
    \blue{Let $\bfV$ be the union $\bigcup_{x \in \bfX(\bR)} \bfU_x$.
	    Let $M:=(uD^{a}_{\bfX})|_{\bfV}$. By construction $M$ is algebraically holonomic.
	    Let $i\colon \bfV \to \bfX$ be the canonical embedding.
	    There is a natural map $uD^{a}_{\bfX} \to i_*(M)$.
	    The kernel of this map are sections of $uD^{a}_{\bfX}$ which vanish on $\bfV$.
	    Since $\bfV$ contains all the real points of $\bfX$ this map is an embedding.
	    This proves the assertion because a submodule of an algebraic holonomic module is algebraically  holonomic.
    }
\end{proof}


\begin{lemma}\label{lem:alg:to:holon}
Let $\bfX$ be a Zariski open subset in an affine space over $\R$ and $X:=\bfX(\bR)$.
Let $u\in \mathscr{D}'(X)$ be a tempered algebraically holonomic distribution. Then $u$ is holonomic.
\end{lemma}
%
\begin{proof}
\blue{We assume $\bfX$ to be a Zariski open subset of $\AA^{n}_{\RR}$.
Let $x\in X$. As $u$ is an algebraic holonomic distribution, by Lemma \ref{lem:hol.eq} there is a Zariski open neighborhood $\bfU\subset \bfX$ of $x$ and a coherent right ideal subsheaf $\cI$ of $D_{\bfU}^{a}$ such that $u_{\mid \bfU(\RR)}\cI=0$ and the quotient sheaf $\cI \backslash D^a_\bfU = 0$.
This implies that the zero sets of the principal symbols of the elements of $\cI$ cut a variety of dimension $\dim \bfU$.
Then using Remark \ref{rem:rightstructure} and Remark \ref{rem:hol}, $u$ is holonomic by Definition \ref{def:holonomic} using $\Omega = \bfU(\CC)$
and the sheaf $\cI_{\CC}\cD_{\CC^n \mid \Omega}$.
}
\end{proof}

The following lemma is obvious from the definitions:

\begin{lemma}\label{lem:hol.loc}
	The notion of algebraically holonomicity is Zariski local. \blue{Let $\bfX$ be a smooth real algebraic variety and $X = \bfX(\R)$.} Namely, if $\bfX=\bigcup_i \bfU_i$ is a Zariski cover and $u \in \mathscr{D}'(X)$  such that $u|_{\bfU_i(\R)}$ is \blue{algebraically} holonomic, then $u$ is \blue{algebraically} holonomic.
\end{lemma}

\begin{lemma} \label{lem:holonomicity:open}
	Let $\bfX$ be a smooth algebraic variety \blue{of dimension $n$} defined over $\R$. Let $x\in X:=\bfX(\R)$ and $\bfU:=\bfX\smallsetminus\{x\}$.
Let $M$ be a \blue{coherent} $D_\bfX^a$-module such that $M|_\bfU$ is holonomic. Then $M$ is \blue{algebraically} holonomic.
\end{lemma}

\begin{proof}
   Let $\bfS:=\Char(M)\subset T^*\bfX$ be the characteristic variety of $M$. Let $p:T^*\bfX\to \bfX$ be the projection. We have:
   \[\bfS\subset \blue{(\bfS\cap p^{-1}(\bfU))}\cup T^*_x\bfX=\Char(M|_\bfU)\cup T^*_x\bfX\]
   therefore
   \[\dim(\bfS)\leq \dim(\Char(M|_\bfU)\cup T^*_x\bfX)=\max(\dim(\Char(M|_\bfU)),\dim (T^*_x\bfX))=n.\]
   But by Bernstein inequality $\dim{\bfS}\geq n$, then $\dim{\bfS}=n$ and $M$ is \blue{algebraically} holonomic.
\end{proof}

\begin{cor} \label{cor:alg:to:holon}
    Let $\bfX,X,x,\bfU,U:=\bfU(\R)$ be as in Lemma \ref{lem:holonomicity:open}. Let $u\in \mathscr{D}'(X)$ be a tempered distribution such that $
    u|_U$ is algebraically holonomic. Then $u$ is holonomic.
\end{cor}
\begin{remark}
    Note that this corollary uses Lemma \ref{lem:coherence}, so its proof is invalid in the analytic case. This is the reason why we use algebraically holonomic distributions.
    See also Remark \ref{rem:holonomicity}.
\end{remark}
\begin{lemma} \label{lem:holonomicity:additivity}
	\blue{Let $\bfX$ be a smooth real algebraic variety and $X = \bfX(\R)$.}
    Algebraically holonomic functions on $X:=\bfX(\R)$ form a left $D^a(\bfX)$-module and tempered algebraically holonomic distributions on $X:=\bfX(\R)$ form a right $D^a(\bfX)$-module.
\end{lemma}
\begin{proof}
    The proof is completely analogous to that of Proposition \ref{prop:sum:holon}
    (in-view of Lemma \ref{lem:hol.eq})
\end{proof}



\begin{lemma}\label{lem:product-holonomicity}
 A product of a $C^\infty$ algebraically holonomic function by an algebraically holonomic distribution is algebraically holonomic.
 Similarly, the product of two $C^\infty$ algebraically holonomic functions is a $C^\infty$ algebraically holonomic function.
\end{lemma}

\begin{proof}
Let $\bfX$ be a smooth real algebraic variety and $X:=\bfX(\RR)$. Let $f\in C^\infty(X)$ be a tempered  function and $u\in \mathscr{D}'(X)$ be a tempered distribution. Assume that both $f$ and $u$ are algebraically holonomic.
Let $M=D^a_\bfX f$ and $N =u D^a_\bfX$. By \cite[Proposition 1.2.9]{HTT} there is a right $D^a_\bfX$-module structure on $M\otimes_{O_\bfX} N$.
It is easy to see that we have an onto map $M\otimes_{O_\bfX} N\to (fu)D_\bfX^a$.
The assertion follows now from the fact that tensor product of algebraic holonomic $D^a_\bfX$-modules is algebraic holonomic (see  \cite[Corollary 3.2.4]{HTT}) and the fact that a quotient of algebraic holonomic $\blue{D^a_\bfX}$-modules is algebraic holonomic (see \cite[Proposition 3.1.2]{HTT}).
The  case of two functions is completely analogous.
\end{proof}

\begin{proposition}[Cartesian product situation]\label{prop:product-situation.alg}
Let $\bfX_1$ and $\bfX_2$ be smooth algebraic varieties defined over $\R$.
Let $u_j$ be a tempered distribution on $X_j:=\bfX_j(\R)$ for $j=1,2$. Suppose that each $u_j$ is algebraically holonomic.
Then the distribution $u_1\otimes u_2$ on $X_1\times X_2$ is also tempered algebraically holonomic.
\end{proposition}

\begin{proof}
    The proof is analogous to the proof of Proposition \ref{prop:product-situation}.
\end{proof}

\begin{lem}[Very basic case]\label{prop:base-case}
Consider the distribution $u$ on $\RR$ given by the $L^1$ function $\RR\to\CC$ sending  $x$ in $(0,1)$ to
\[
	x^a(\log x) ^b \exp(\mathrm{i} x^c)
\]
and other $x$ to $0$, with complex $a$ with $\Re(a)>-1$, integer $b\ge 0$, and integer $c$.
Then $u$ is tempered and algebraically holonomic.
\end{lem}

\begin{proof}
$ $
\begin{enumerate}[Step 1.]
    \item Proof that $u$ is algebraically holonomic outside $\{0,1\}$.\\
    Let $v\in C^\infty(\R\smallsetminus \{0,1\})$ be a function such that $u=v \mathrm{d}x$, where $\mathrm{d}x$ is the Haar measure on the line.
    By Lemma \ref{lem:product-holonomicity} it is enough to show that $v$ is algebraically holonomic.
    \begin{enumerate}[{Case} 1.]
    \item $a=1; b=c=0$\\
    Obvious.
    \item $b=1; a=c=0$\\
    Follows from the relation:
    \[\left (\frac{d}{dx}x\frac{d}{dx}\right)(\log(x))=0.\]
    \item $a=b=0$\\
    Follows from the relation:
    \[\left(x^{1-c}\frac{d}{dx}-\mathrm{i}c\right)(e^{\mathrm{i}x^c})=0.\]
    \item General case.\\
    Follows from the previous cases using Lemma \ref{lem:product-holonomicity}.
\end{enumerate}
    \item Proof that $u$ is algebraically holonomic.\\
    Follow from the previous step using Corollary \ref{cor:alg:to:holon}.
\end{enumerate}

That $u$ is tempered is clear.
\end{proof}

\begin{prop}[Special base case]
\label{prop:base-case-u}
The distribution $u$ on $\RR^n$ given by the $L^1$ function $\RR^n\to\CC$ sending $x$ in $(0,1)^n$ to
\begin{equation}\label{eq:u1}
	(\prod_{j=1}^n x_j^{a_j}) (\prod_{j=1}^n (\log x_j)^{b_j}  ) \exp(\mathrm{i} 
\prod_{j=1}^n x_j^{c \delta_j })
\end{equation}
and  other $x$ to $0$, with complex $a_j$ with $\Re(a_j)>-1$, integers $b_j\ge 0$, integer $c<0$, and integers $\delta_j\in \{0,1\}$,
is tempered and algebraically holonomic.
\end{prop}
\begin{proof}[Proof of Proposition \ref{prop:base-case-u}]
$ $
We prove the statement by induction on $n$.
We denote a distribution defined as $u$ but with other constants $a_j',b_j', \delta_j', c'$ by $u_{\{a_j',b_j', \delta_j', c'\}}$.
The base of the induction follows from  Lemma \ref{prop:base-case}. We will assume that $n\geq 2$.
\begin{enumerate}[{Case} 1.]
    \item $\delta_j=0$ for some $j$.\\
We are done by induction and Proposition \ref{prop:product-situation.alg}.
\item $\delta_j=1$ for all $j$.



By Lemmas \ref{lem:hol.loc} and \ref{lem:holonomicity:open}
it is sufficient to prove holonomicity in the open set $\bfU_j\subset \blue{\A^n_{\RR}} $ given by $x_j\neq 0$ for each $j$.  Up to reordering the coordinates, we may suppose that $j=2$.
We will show that for any tuple $b_j'$ the distribution $$u_{\{a_j,b_j',\delta_j, c\}}|_{\bfU_2(\R)}$$ is holonomic. We will do it by induction on $b_1'$.
    Consider the map $\phi:\bfU_2\to \bfU_2$ given by $$\phi(x_1,\dots,x_n)= (x_1x_2,x_2,\dots,x_n).$$
    Let $$u':=u_{\{a_j,b_j', \delta_j',c\}}|_{\bfU_2(\R)},$$ where $\delta_2'=0$ and $\delta_j'=1$ \blue{if} $j\neq 2$.  By Case 1, the distribution $u'$ is algebraically holonomic.
    So, $\phi^*(u')$ is algebraically holonomic.
The distribution $\phi^*(u')$ is given in the similar way as $u$ but with the formula:
\begin{equation}\label{eq:u}
	(x_1^{a_1}x_2^{a_2+a_1}\prod_{j=3}^n x_j^{a_j}) (\log x_1+\log x_2)^{\blue{b_1'}}(\prod_{j=2}^n (\log x_j)^{\blue{b_j'}})
\exp(\mathrm{i}\prod_{j=1}^n x_j^{c}).
\end{equation}
It is easy to see that $x_2^{-a_1}$ is an algebraically holonomic function on $\bfU_2(\R)$. Thus, by Lemma \ref{lem:product-holonomicity}, $x_2^{-a_1}\phi^*(u')$ is algebraically holonomic.

    By the binomial formula, we obtain
    $$x_2^{-a_1}\phi^*(u')=\sum_{k=0}^{\blue{b_1'}} \blue{\binom{b_1'}{k}}u_{\{a_j,b_{jk}', \delta_j,c\}}|_{\bfU_2(\R)}$$
    where: \blue{for each $k$, we have $b'_{1k}=b_1'-k$, $b'_{2k}=b_2'+k$, and $b'_{jk}=b_j'$ for each $j>2$.}

    By the induction assumption \blue{on $b_1'$} and by summation (see Lemma \ref{lem:holonomicity:additivity}), we get that
    $$\sum_{k=1}^{\blue{b_1'}} \blue{\binom{b_1'}{k}}u_{\{a_j,b_{jk}', \delta_j,c\}}|_{\bfU_2(\R)}$$ is algebraically holonomic.
We obtain:
    \[
	    (u_{a_j,b_j',\delta_j,c})|_{\bfU_2(\R)}=
	    x_2^{-a_1}\phi^*(u')-\sum_{k=1}^{\blue{b_1'}} \blue{\binom{\blue{b_1'}}{k}}u_{\{a_j,b_{jk}', \delta_j,c\}}|_{\bfU_2(\R)}.
    \]
Applying Lemma \ref{lem:holonomicity:additivity} again, we get that $(u_{a_j,b_j',\delta_j,c})|_{\bfU_2(\R)}$ is algebraically holonomic, as requested.





\end{enumerate}
That $u$ is tempered is again clear.
\end{proof}

\begin{proof}[Proof of Proposition \ref{cor:base-case-u}]
The special case of Proposition \ref{cor:base-case-u} which has furthermore all $c_j$ equal to $\delta_jc$ for some integer $c\le 0$ and some $\delta_j\in\{0,1\}$ follows from Proposition \ref{prop:base-case-u} and Lemma \ref{lem:alg:to:holon}. The general case is reduced to this situation by Propositions \ref{prop:power-map:holon} and \ref{prop:push:holon}.
\end{proof}

\section{Distributions of $\cCexp$-class}\label{sec:uCexp}

In this section, we introduce the notion of $\cCexp$-class distributions on $\RR^n$, and more generally on an open set $U\subset \R^n$. We do this by considering the continuous wavelet transforms of the distributions. 

	\begin{defn}[Mother wavelet] \label{notation:motherwavelet}
		Let $\psi:\RR\to\CC$ be a $\cCexp$-class function which is $C^\infty$ and compactly supported and with $\int_{\RR} \psi(t)dt = 1$, $\psi \geq 0$, and $\psi(x)=\psi(-x)$, for all $x \in \RR$. Suppose furthermore that the support of $\psi$ is subanalytic. Such a function exists by Proposition \ref{prop:cutoffCexp} and its proof.
	     We define
	     $$\Psi(x_1,\dots,x_n)=\prod_{i=1}^{n} \psi(x_i)$$
and call it a mother wavelet on $\RR^n$.

	\end{defn}

From here on, we fix a function $\psi$ and the corresponding mother wavelet $\Psi$ on $\RR^n$, as in Definition \ref{notation:motherwavelet}.

%
%
	\begin{defn}[Continuous wavelet transform]\label{defn:wevelet}
Let $u$ be a distribution on an open set $U\subset\RR^n$.
By the continuous wavelet transform of $u$ with respect to the mother wavelet $\Psi$ we mean the complex-valued function on
$\RR^n \times \RR_{>0}$, denoted by $\cW(u)$ or $\cW_{\Psi}(u)$, sending $(x,\lambda)$ with $\lambda>0$ to $u$ evaluated in the test function
$$
y \mapsto \Psi\left(\lambda\cdot (x-y)\right)
$$
if the support of that test function is included in $U$, and to $0$ otherwise.
\end{defn}

	\begin{defn}[Distributions of $\cCexp$-class]\label{def:distrCexp}
		A distribution $u$ on an open set $U\subset \RR^n$ is a distribution of $\cCexp$-class if its continuous wavelet transform  $\cW(u)$ with respect to the mother wavelet $\Psi$
		is a $\cCexp$-class function. Write $\mathscr{D}'_{\cCexp}(U)$ for the $\cCexp$-class distributions on $U$.
	\end{defn}

Note that by the definitions, if $u$ is a distribution of $\cCexp$-class on an open set $U\subset \R^n$, then $U$ is a subanalytic set.
We show first that Definition \ref{def:distrCexp} is independent of the choice of mother wavelet in Definition \ref{notation:motherwavelet}, see Corollary \ref{prop:ind}. We also find that the Fourier transform of a tempered distribution of $\cCexp$-class is again of $\cCexp$-class, see Proposition \ref{prop:Four}. Any locally integrable $\cCexp$-class function is a $\cCexp$-class distribution, see Proposition \ref{prop:L1}.


\begin{prop}\label{prop:gen}\label{prop:Julia}
Let $u$ be a distribution of $\cCexp$-class on an open set $U\subset \R^n$.
Consider a collection $\phi_z$ of functions in $C^\infty_c(U)$, for $z\in \RR^m$, such that the total family
$$
\RR^m\times U \to\CC: (z,x) \mapsto \phi_z(x)
$$
is of $\cCexp$-class. Then the function
$
z\mapsto   \la u,\phi_z \ra
$
is of $\cCexp$-class.

If $u$ is furthermore tempered, if $\psi_z$ is in $\cS(U)$ for each $z\in\R^m$ and if the total family
$$
\RR^m\times U \to\CC: (z,x) \mapsto \psi_z(x)
$$
is of $\cCexp$-class, then the function
$
z\mapsto  \la u,\psi_z \ra
$
is of $\cCexp$-class.
\end{prop}
\begin{proof} 
	For any $\eps>0$, consider the function
	$$\Psi_{\eps}:x \mapsto \frac{1}{\eps^n}\Psi\left( \frac{x}{\eps} \right).$$
	For any $z \in \RR^m$, the function $\Psi_{\eps}*\phi_z$ belongs to $C_{c}^{\infty}(\RR^n)$.
	Furthermore recall that
	$$\supp\left( \Psi_{\eps}*\phi_z \right) \subset \supp\left( \Psi_{\eps} \right)+\supp\left( \phi_z \right).$$
	Note that $\supp\left(\Psi_{\eps}\right)\to \{0\}$ when $\eps \to 0$, so there is $\eps_0>0$, depending  on $z$, such that for any
        $\eps \in ]0,\eps_0[$,
	$$\supp\left( \Psi_{\eps}*\phi_z \right) \subset \supp\left( \Psi_{\eps} \right)+\supp\left( \phi_z \right) \subset U.$$
	meaning that the function $\Psi_{\eps}*\phi_z$ belongs to $C_{c}^{\infty}(U)$. 
	
 	Using a classical regularization result and continuity of distributions (see for instance \cite[Theorems 1.3.2 and 2.1.4]{Hormander83}), we have, for any $z \in \RR^m$, the following limit behavior
	\begin{equation} \label{eq:limreg}
		\la u,\Psi_{\eps}*\phi_z \ra \underset{\eps \to 0}{\to} \la u,\phi_z \ra.
	\end{equation}
	By definition of the convolution product we have for small enough $\eps$, depending on  $z$,
	$$\la u,\Psi_{\eps}*\phi_z \ra = \la u,y\mapsto \la \phi_z,t\mapsto 1_{\Supp{\phi_z}}(t)\Psi_{\eps}(y-t)\ra \ra,$$
	so that
	$$\la u,\Psi_{\eps}*\phi_z \ra = \la u\otimes \phi_z,(y,t)\mapsto 1_{\Supp{\phi_z}}(t)\Psi_{\eps}(y-t) \ra =
	\la\phi_z,t\mapsto \la u,y\mapsto \Psi_{\eps}(y-t) \ra \ra.$$
	Thus, (using that $\psi(-x)=\psi(x)$), we have for small enough $\eps$, depending on  $z$,
	\begin{equation} \label{eq:gen1} \la u,\Psi_{\eps}*\phi_z \ra = \frac{1}{\eps^n} \la \cW(u)(1/\eps,t),\phi_z \ra.
	\end{equation}
	As $u$ is of $\cCexp$-class, $(\eps,t)\mapsto \cW(u)(1/\eps,t)$ is a $\cCexp$-class function. Also $(z,t)\mapsto \phi_z(t)$ is a $\cCexp$-function (by assumption), and hence, by the integration result from Theorem \ref{thm:stab}, we deduce that
	$$(\eps,z)\mapsto \frac{1}{\eps^n} \la \cW(u)( 1/\eps,t),\phi_z \ra
	$$
	is a $\cCexp$-class function on $\R_{>0}\times \R^m$. Now, by (\ref{eq:limreg}) and (\ref{eq:gen1}), we have for each $z$
\begin{equation} \label{eq:gen2}
	\frac{1}{\eps^n} \la \cW(u)(1/\eps,t),\phi_z \ra  \underset{\eps \to 0}{\to} \la u,\phi_z \ra.
\end{equation}
	Since limits of $\cCexp$-class functions are still of $\cCexp$-class by Proposition \ref{prop:limits}, we deduce from (\ref{eq:gen2}) that $z\mapsto \la u,\phi_z \ra$ is of $\cCexp$-class.
In the tempered case one proceeds similarly.
\end{proof}

\begin{cor}\label{prop:ind}
Definition \ref{def:distrCexp} is independent of the choice of mother wavelet $\Psi$. That is, if $\psi'$ is any other $\cCexp$-class function which is $C^\infty$ and compactly supported and with $\int_{\RR} \psi'(t)dt = 1$, $\psi' \geq 0$, $\psi'(x)=\psi'(-x)$ and with subanalytic support, and if we define the mother wavelet $\Psi'$ on $\R^n$ corresponding to $\psi'$ (instead of $\psi$), then the following holds for any distribution $u$ on $U$. The continuous wavelet transform with respect to $\Psi$ is a $\cCexp$-class function if and only if the continuous wavelet transform with respect to $\Psi'$ is a $\cCexp$-class function.
\end{cor}
\begin{proof}
	Follows at once from Proposition \ref{prop:gen} applied to $z=(x,\lambda) \in \RR^n\times \RR_{>0}$ and $\phi_{(x,\lambda)}:y \mapsto \psi'(\lambda.(x-y))$ and $\phi:( (x,\lambda),y)\mapsto \phi_{(x,\lambda)}(y)$ which are $\cCexp$-functions.
\end{proof}

\begin{prop}[Fourier transforms]\label{prop:Four}
Let $u$ be a distribution of $\cCexp$-class on $\RR^n$. Suppose that $u$ is tempered. Then the Fourier transform of $u$ is also a tempered distribution of $\cCexp$-class.
\end{prop}
\begin{proof}
This follows from Proposition \ref{prop:gen}.
Indeed, for any $\varphi \in \cS(\RR^n)$, we have by definition
	$$\la \cF u,\varphi \ra = \la u,\cF\varphi\ra.$$

Hence, we have the equality of functions
	$$(\lambda,x) \mapsto \la \cF u,y\mapsto \Psi(\lambda(x-y)) \ra = (\lambda,x)\mapsto \la u, \cF(y\mapsto \Psi(\lambda(x-y))) \ra.$$
	As $(\lambda,y,x)\mapsto \Psi(\lambda(x-y))$ is a $\cCexp$-class function, by Theorem \ref{thm:stab},
	$$(\lambda,x)\mapsto \cF(y\mapsto \Psi(\lambda(x-y)))$$
	is also a $\cCexp$-class function, and we conclude by the second part of Proposition \ref{prop:gen}.
\end{proof}

\begin{prop}\label{prop:L1}
If $f$ is a locally integrable $\cCexp$-class function on an open set $U\subset \RR^n$, then it gives a distribution of $\cCexp$-class on $U$.
\end{prop}
\begin{proof}
This follows from Theorem \ref{thm:stab} and Definition \ref{def:Cexp}. 
\end{proof}


Next we study push-forward and pull-back along appropriate maps and derivation for $\cCexp$-class distributions.

\begin{prop}[Push-forward]\label{prop:push}
Push-foward of distributions along subanalytic maps preserves the $\cCexp$-class. In detail, let $u$ be a $\cCexp$-class distribution on an open set $U\subset \R^n$, and let $f:U\to V$ be a subanalytic function which is proper and analytic on (an open subanalytic neighborhood in $U$ of) the support of $u$, for some open set $V\subset \R^m$. Then $f_*(u)$ is a distribution of $\cCexp$-class on $V$.
\end{prop}
\begin{proof}
Consider a proper smooth analytic subanalytic map $f:U \to V$ between two open sets $U \subset \RR^n$ and $V \subset \RR^m$. Let $u$ be a $\cCexp$-class distribution on $U$. The push-forward $f_*u$ is defined for any $\varphi \in C_{c}^{\infty}(V)$ by
	$$\la f_*u,\varphi \ra = \la u,\varphi \circ f \ra.$$
We have the equality
	$$(\lambda,x)\mapsto \la f_*u,y\mapsto \Psi(\lambda(x-y)) \ra = \la u,z \mapsto \Psi(\lambda(x-f(z)))\ra.$$
	As $(\lambda,z,x)\mapsto \Psi(\lambda(x-f(z)))$ is a $\cCexp$-function by Lemma \ref{prop:compCexp}, the proposition follows from Proposition \ref{prop:gen}.
\end{proof}


In the next two results, we use the notion of wave front set $\WF(u)$  of a distribution $u$ from \cite[Section 8]{Hormander83}.

\begin{prop}[Pull-back]\label{prop:pull}
	Pull-back of distributions preserves $\cCexp$-class, along subanalytic maps, as follows.  Let $X$ and $Y$ be two subanalytic open subsets of $\RR^m$ and $\RR^n$ respectively and let $f:X\to Y$ be an analytic subanalytic map. Denote the set of normals of $f$ by
	$$N_f = \{(y,\eta)\in Y \times \RR^n \mid \exists x\in \R^m, \:^{t}d f(x)\eta = 0 \mbox{ and } f(x)=y\}.$$
	For any distribution $u$ on $Y$ of $\cCexp$-class with wave front set $\WF(u)$ satisfying the condition $N_f \cap \WF(u)=\emptyset$,
	the pull-back $f^*(u)$ is well defined and of $\cCexp$-class.
\end{prop}
	\begin{rem}\label{rem:pull}
		If $u$ is a smooth function then $f^*(u)$ is defined as $u\circ f$.
		By continuity of $f^*$, we get that $f^*u$ is $u \circ f$ for a locally integrable function $u$ such that
		$N_f \cap \WF(u)=\emptyset$, see \cite[Theorem 8.2.3, Theorem 8.2.4]{Hormander83}.
	\end{rem}

\begin{proof}[Proof of Proposition \ref{prop:pull}] 
The pull-back $f^*(u)$ is well defined by application of \cite[Theorem 8.2.4]{Hormander83}.
	We consider the following exhaustion by subanalytic compact sets $(K_{\eps})_{\eps>0}$ of $Y$ with for any $\eps>0$
	$$K_{\eps} = \{ y \in \RR^n \mid ||y||\leq 1/\eps,\: d(y,\RR^n\smallsetminus  Y)\geq \eps\},$$
with $d$ the Euclidean distance and $||y||$ the Euclidean norm.
	We have $Y=\cup_{\eps>0} K_{\eps}$, for any $\eps'\geq \eps$, $K_{\eps'}\subset K_{\eps}$. Furthermore, for any compact $K$ of $Y$, there is $\eps(K)$ such that for each $\eps \leq \eps(K)$ one has  $K\subset K_\eps$. By Propostion \ref{prop:cutoffCexp}, for any $\eps>0$, we may consider a cut-off function $\chi_{\eps} \in C_{c}^{\infty}(Y)$ of $\cCexp$-class such that $\chi_{\eps}=1$ on a neighborhood of $K_\eps$.
	Furthermore, it follows 
from the constructions of the family $(K_{\eps})$ and the $\chi_{\eps}$ (as for Propostion \ref{prop:cutoffCexp}) that the function $\chi:(\eps,y)\mapsto \chi_{\eps}(y)$ is of $\cCexp$-class.
	
	With notations of the proof of Proposition \ref{prop:gen}, for any $\eps>0$, we consider
	$$u_\eps = (\chi_\eps u)*\Psi_{\eps} = x\mapsto \la u,y \mapsto \chi_\eps(y)\Psi_{\eps}(x-y)\ra.$$
	By Proposition \ref{prop:gen}, $u_\eps$ is a $\cCexp$-class function (this uses that $u$ is a $\cCexp$-class distribution).
	Furthermore, by Lemma \ref{prop:compCexp},  $u_\eps \circ f$ is also a $\cCexp$-class function.
	As in the proof for Proposition \ref{prop:gen} and by \cite[Theorem 8.2.3]{Hormander83} and using its notations, with $\Gamma=\WF(u)$, for $\eps$ small enough one has $u_\eps \in C_{c}^{\infty}(Y)$ and $u_\eps \underset{\eps \to 0} \to u$ in $\cD'_{\Gamma}(Y)$. By continuity of $f^{*}$ in \cite[Theorem 8.2.4]{Hormander83}, one has
	$$f^{*}(u_\eps)=u_\eps \circ f \underset{\eps \to 0}{\to} f^{*}(u).$$
	In particular, for any $(\lambda,x)$ one has
	$$\la f^{*}(u_\eps),t \mapsto \Psi(\lambda(x-t))\ra \underset{\eps \to 0}{\to} \la f^{*}(u),t \mapsto \Psi(\lambda(x-t))\ra.$$
	As $(\lambda,x,\eps)\mapsto \la f^{*}(u_\eps),t \mapsto \Psi(\lambda(x-t))\ra$ is a $\cCexp$-function, by Proposition \ref{prop:limits} and by taking the limit $\eps\to 0$, we get that $(\lambda,x)\mapsto \la f^{*}(u_\eps),t \mapsto \Psi(\lambda(x-t))\ra$ is a $\cCexp$-class function, meaning that $f^{*}(u)$ is a $\cCexp$-class distribution.
\end{proof}

We include the following result for independent interest.

\begin{prop}\label{lem:WF}
Let $u$ be a distribution on an open set $U$ in $\RR^n$.
If $u$ is holonomic, then $u$ is WF-holonomic in the sense of \cite{AizDr}.
\end{prop}
\begin{proof}
		If $u$ is holonomic then by Definition \ref{def:holonomic}, for any $x \in U$, there is an open neighborhood $\Omega$ of $x$ in $\CC^n$ and a coherent left ideal sheaf $\cI$ of  $\cD_{\CC^n \mid \Omega}$, such that $\cI.u=0$ on $\Omega$ and the dimension of the characteristic variety $\Char\left(\cD_{\CC^n \mid \Omega}/\cI\right)$ is $n$. In particular by a standard result on D-modules (see \cite[Section IV.1]{Granger-Maisonobe}), there are finitely many analytic submanifolds $(A_i)_{i \in I}$ of $\Omega$, such that
		\[\Char\!\left(\cD_{\CC^n \mid \Omega}/\cI\right) \subset \bigcup_{i \in I} \CN_{A_i}^{\Omega}\]
where for any $i \in I$, $\CN_{A_i}^{\Omega}$ is the conormal bundle of $A_i$ in $\Omega$.
		By \cite[Theorem 8.3.1]{Hormander83}, we have
		\[\WF(u) \subset \bigcup_{i \in I} \CN_{A_i \cap \RR^n}^{U}\]
		which implies by definition that $u$ is WF-holonomic in the sense of \cite{AizDr}.
\end{proof}

\begin{prop}[Derivation]\label{prop:der}
	Any partial derivative of distributions preserves the $\cCexp$-class. 
\end{prop}

\begin{proof}
Let $u$ be a $\cCexp$-class distribution on an open set $U$ of $\RR^n$. By definition, for any $\varphi \in C_{c}^{\infty}(U)$,
	we have
	$$\la \partial_{y_i} u,\varphi \ra =- \la u,\partial_{y_i}\varphi\ra.$$
Thus, one has
$$(\lambda,x)\mapsto \la \partial_{y_i}u,y\mapsto \Psi(\lambda(x-y)) \ra = - \la u,\partial_{y_i}(y \mapsto \Psi(\lambda(x-y)))\ra.$$
	Applying the definition of a partial derivative (as a limit) and Proposition \ref{prop:limits} on limits of $\cCexp$-class functions, we get that $(\lambda,y,x)\mapsto \partial_{y_i}(y \mapsto \Psi(\lambda(x-y)))$ is a $\cCexp$-class function which is furthermore smooth with compact support, and hence, the proposition follows from Proposition \ref{prop:gen}.
\end{proof}

Proposition \ref{prop:anti-Cexp} will give that also anti-derivatives of $\cCexp$-distributions can be taken inside the $\cCexp$-class.
We end this section with a further link between $\cCexp$-class functions and $\cCexp$-class distributions.



%


\begin{proposition}\label{prop:Cexp.and.cont.is.CexpFun}
If a distribution $u$ on an open set $U\subset \RR^n$ is of $\cCexp$-class and is given by a continuous function $f$, then $f$ is of $\cCexp$-class. In symbols:
$$
\mathscr{D}'_{\cCexp}(U)\cap C^0(U) \subset \Cexp(U).
$$
More generally, if a distribution $u$ on an open set $U\subset \RR^n$ is of $\cCexp$-class and given by a locally integrable function, then there exists a locally integrable function $f$ in $\cCexp(U)$ such that $\la u,\phi \ra = \int_{x\in \R^n}f(x)\phi(x)dx$ for all functions $\phi$ in $C_{c}^{\infty}(U)$.
\end{proposition}
\begin{proof}
Using the notation from the proof of Proposition \ref{prop:gen}, the function which sends $(\eps,x)$ in $\R_{>0}\times U$ to
$$\la u,y\mapsto \Psi_{\eps}(x-y) \ra =(u*\Psi_{\eps})(x)$$
whenever $y\mapsto \Psi_{\eps}(x-y)$ has support contained in $U$ and to zero otherwise
is of $\cCexp$-class on $\RR_{>0}\times U$ (this uses that $\Supp(\Psi)$ and $U$ are subanalytic), and converges to a function $f_0$ on $U$  when $\eps\to 0$, by \cite[Theorem 1.3.2]{Hormander83}. In the first case, one has $f=f_0$ and $f$ is of $\cCexp$-class by Proposition \ref{prop:limits}. In the second case, there is a function $f$ of $\cCexp$-class such that $f=f_0$ holds almost everywhere, by construction and by Proposition \ref{prop:limits}.
\end{proof}

%







\section{Anti-derivatives}\label{sec:antider}

We will use a construction of anti-derivatives of distributions to reduce to locally integrable distributions.

Throughout this section we fix a test function $\rho \in C_c^{\infty}\cap \cCexp(\R)$ on $\R$ with integral $1$ as in Proposition \ref{prop:cutoffCexp}, with the closed unit ball around $0$ for $K$.


\subsection{Construction of the anti-derivative}
\begin{notation} \label{not:ad}
	For any $i \in \{1,\dots, n\}$,
\begin{itemize}
\item define $C_{i}:C_c^{\infty}(\R^n)\to C_c^{\infty}(\R^n)$  by
	$$C_i(g)(x_1,\dots,x_n)=g(x_1,\dots,x_n) -\rho(x_i) \int_{-\infty}^{\blue{+\infty}} g(x_1,\dots,x_{i-1},y,x_{i+1}\dots,x_n)\mathrm{d}y.$$
\item define $B_{i}:C_c^{\infty}(\R^n)\to C^{\infty}(\R^n)$  by
	$$B_i(g)(x_1,\dots,x_n)=-\int_{-\infty}^{x_i} g(x_1,\dots,x_{i-1},y,x_{i+1}\dots,x_n)\mathrm{d}y.$$
\item define $A_{i}'=B_i\circ C_i$. Note  that $\mathrm{Im} A_i' \subset  C_c^{\infty}(\R^n)$. So, we consider $A_i'$ as a map from $C_c^{\infty}(\R^n)$ to itself.
      \item 
      define  $A_{i}:\scD'(\R^n)\to \scD'(\R^n)$  by
	      $$\left \langle A_{i}(u),f  \right  \rangle=\left  \langle u,B_{i} (C_{i}(f)) \right \rangle = \langle u,A_{i}'(f) \rangle$$
	      for any $f \in C_c^{\infty}(\R^n)$.
      \item define $A:=A_1\circ\cdots\circ A_n$ and similarly  $A':=A_1'\circ\cdots\circ A_n'$.
\end{itemize}
\end{notation}
\begin{lemma}\label{lem:anty.char}
Consider a distribution $u$ on $\R^n$. The following hold.
\begin{enumerate}
\item\label{item:1:anty} We have  $\frac{\partial}{\partial x_i} A_{i}(u)= u$.
\item For any $f \in C_c^{\infty}(\R^{n-1})$ we have $\la A_{i}(u),\rho_{i,f}\ra=0$ where
	$$ \rho_{i,f}(x_1,\dots,x_n)=\rho(x_i)f(x_1,\dots x_{i-1},x_{i+1},\ldots,x_n).$$
\end{enumerate}
\end{lemma}
\begin{proof}For any $f \in C_c^{\infty}(\R^n)$ we have
$$ \la \frac{\partial}{\partial x_i} A_{i}u,f\ra= \la A_{i}u,-\frac{\partial}{\partial x_i} f\ra=\la u,B_{i} \circ C_{i}\lp -\frac{\partial}{\partial x_i} f\rp \ra=\la u,B_{i} \lp-\frac{\partial}{\partial x_i} f\rp\ra$$
which equals $\la u,f\ra$ as desired. The second item is clear from the observation that $C_i(\rho_{f,i})=0$.
\end{proof}

\begin{proposition}\label{prop:anti}
	Let $u$ be a distribution on  $\R^n$. If $A_i u$ is holonomic then $u$ is holonomic.
\end{proposition}
\begin{proof}
By construction, we have $u=\partial. A_i u$, so by Proposition \ref{prop:sum:holon} if $A_i u$ is holonomic, then $u$ is holonomic.
\end{proof}

\begin{proposition}\label{prop:anti-Cexp}
Let $u$ be a distribution on  $\R^n$. Suppose that $u$ is of $\cCexp$-class. Then also $A_i u$ is of $\cCexp$-class.
\end{proposition}
\begin{proof}
The proposition follows using an anti-derivative $\widetilde{\Psi_i}$ of $\Psi$ using $B_i\circ C_i$ from Notation \ref{not:ad} and by Proposition \ref{prop:gen}. In detail, note that $\widetilde{\Psi_i}:=B_i\circ C_i(\Psi)$ with $\Psi$ the mother wavelet is of $\cCexp$-class by Theorem \ref{thm:stab}. We reason on the continuous wavelet transform of $A_iu$ at $(\lambda,x)\in\R_{>0}\times \R^{n}$, using the partial derivation $\partial_{y_i}=\partial/\partial y_i$, and write
	$$\la A_i u, y\mapsto \Psi(\lambda(x-y)) \ra  = \la A_i u, y \mapsto \partial_{y_i}\left( -\frac{1}{\lambda}\widetilde{\Psi}_i(\lambda(x-y))\right)\ra,$$
which holds by a similar calculation as for the proof of Lemma \ref{lem:anty.char} and the chain rule.
	We thus have by (1) of Lemma \ref{lem:anty.char}
	$$\la A_i u, y\mapsto \Psi(\lambda(x-y))\ra = \la u,y \mapsto \frac{1}{\lambda}\widetilde{\Psi}_i(\lambda(y-x))\ra$$
	which is of $\cCexp$-class when seen as a function in $(\lambda,x)$ by Proposition \ref{prop:gen}. By Definition \ref{def:distrCexp}, the above means precisely that $A_i u$ is of $\cCexp$-class.
\end{proof}

\subsection{Anti-derivatives can make distributions into continuous functions}
In this subsection we will prove the following result.
\begin{proposition}\label{prop:anty.make.cont}
	Let $u$ be a distribution on  $\R^n$ with compact support. Then there exists $k$ in $\NN$ s.t. $A^k u$ belongs to $C^0(\R^n)$.
\end{proposition}

\begin{remark}
Proposition \ref{prop:anty.make.cont} holds more generally for $u$ in $\cS'(\RR^n)$ instead of $\scD'_c(\R^n)$, but we will not need this fact and we don't show it here.
\end{remark}

This type of statements is rather standard (see for instance \cite[Théorème XXVI]{Schwartz}) and used for example in classical proofs of the regularity theorem for solutions of elliptic equations.
For completeness we include a proof here. One of the most common techniques to prove such a statement is the notion of Sobolev spaces. Since we are not interested in sharp values for $k$, we choose a more elementary approach instead.


\begin{notation}
For a compact $K\subset \R^n$ and any $k\ge 0$ and $p>0$ denote $$C^{k}_K(\R^n)=\{f\in C^{k}(\R^n)\mid \Supp(f)\subset K\}$$ and $$L^{p}_K(\R^n):=\{f\in L^{p}(\R^n)\mid \Supp(f)\subset K\},$$ where $\Supp(f)$ stands for the support of $f$.
\end{notation}
Let $K$ be $K_0^n\subset \R^n$ where $K_0\subset \R$ is a compact interval that contains the support $\Supp(\rho)$ of $\rho$.
The following lemma is straight-forward.
\begin{lemma}\label{lem:anty.dual.make.L1.sm}
For $K$ being $K_0^n$ as above, one has
\begin{enumerate}
\item $A'(C^\infty_K(\R^n)) \subset C^\infty_K(\R^n)$.
\item $A'$ can be continuously extended to operators:
\begin{enumerate}
	\item\label{lem:anty.dual.make.L1.sm:1} $A': C^m_K(\R^n)\to C^{m+1}_K(\R^n)$, for all integer $m\geq 0$.
\item\label{lem:anty.dual.make.L1.sm:2} $A': L^\infty_K(\R^n)\to C^{0}_K(\R^n)$.
\item\label{lem:anty.dual.make.L1.sm:3} $A': L^1_K(\R^n)\to L^\infty_K(\R^n)$.
\end{enumerate}
\end{enumerate}
\end{lemma}
\begin{lemma}\label{cor:anty.L1} \label{cor:anty.make.L1.cont}
	For any $i \in \{1,\dots,n\}$, we have
	\begin{equation}\label{keyformula}
		A_i|_{L^{1}_{K}(\R^n)}=({\rm Id}-E_i)\circ D_i
	\end{equation}
	where for any $u \in L^{1}_{K}(\R^n)$, for any $x \in \RR^n$,
	\[D_i(u)(x_1,\dots,x_n):=\int_{-\infty}^{x_i} u(x_1,\dots, x_{i-1},y,x_{i+1},\dots,x_n)\mathrm{d}y\]
	and
	\[E_i(u)(x_1,\dots,x_n):=\int_{-\infty}^{\blue{+\infty}} u(x_1,\dots, x_{i-1},y,x_{i+1},\dots,x_n)\rho(y)\mathrm{d}y.\]
	Here, 
	${\rm Id}$ stands for the identity map.
Furthermore, the following hold.
	\begin{enumerate}
		\item \label{cor:anty.compact} If $u$ is in $L^{1}_{c}(\R^n)$, then also $A(u)$ has compact support.
		\item \label{cor:anty.make.L1.cont:1} $A(L^{1}(\R^n)) \subset L^{\infty}(\R^n)$.
		\item \label{cor:anty.make.L1.cont:2} $A(L^{\infty}_{K}(\R^n)) \subset C^0(\R^n)$.
	\end{enumerate}
\end{lemma}
\begin{proof}
	The equality (\ref{keyformula}) follows from a direct computation.
	\begin{enumerate}
		\item
			If $u$ is $L^{1}_{c}(\RR^n)$ with compact support in some $[-R,R]^n$ then $A_i(u)$ has compact support in $[-R,R]^n$. Indeed, let $x \in \RR^n$. If there is $j\neq i$ with $\abs{x_j}>R$, then, $A_i(u)(x)=0$ by integration of the zero function. If for all $j\neq i$, $\abs{x_j}\leq R$ and $\abs{x_i}>R$ then, we have
			\[
				\begin{array}{lcl}
					A_i(u)(x) & = & \int_{-R}^{R}u(x_1,\dots, x_{i-1},y,x_{i+1},\dots,x_n)\mathrm{d}y \\ \\
					& - & \int_{-\infty}^{\infty} \left(\int_{-R}^{R}u(x_1,\dots, x_{i-1},z,x_{i+1},\dots,x_n)\mathrm{d}z\right)\rho(y)\mathrm{d}y = 0
				\end{array}
			\]
			because $\int_{\RR} \rho(y)dy = 1$. 			
			Then, by composition we get that $A(u)$ has compact support contained in $[-R,R]^n$. Property (\ref{cor:anty.compact}) of the lemma follows.
		\item
			Note that $E_i$ commutes with $D_j$ if $i\neq j$.
			Note also that $E_i$ preserves the spaces $\blue{L^{\infty}(\R^n)}$ and $ C^0(\R^n)$ so it is enough to show statements
			\eqref{cor:anty.make.L1.cont:1} and \eqref{cor:anty.make.L1.cont:2} when $A$ is replaced with $D:=D_1\circ\cdots \circ D_n$. But this follows from: for any $u \in L^1(\RR^n)$
		$$D(u)(x_1,\dots,x_n)=\int_{-\infty}^{x_1}\dots \int_{-\infty}^{x_n}  u(y_1,\dots,y_n)\blue{\mathrm{d}y_n\dots \mathrm{d}y_1},$$
			which is a bounded function by integrability of $u$.
			Furthermore for any $u \in L^{\infty}_{K}(\RR^n)$, \blue{assuming $K \subset [-R,R]^n$ with $R>0$}, we can check that $D(u)$ is continuous: for any $x,x' \in \RR^n$, we have
			\[\abs{D(u)(x)-D(u)(x')}\leq \blue{(2R)^{n-1}\sum_{i=1}^{n}\abs{x_i-x_i'}\norm{u}_{\infty}}.\]
Thus, properties (\ref{cor:anty.make.L1.cont:1}) and (\ref{cor:anty.make.L1.cont:2}) of the lemma are proved and the proof is complete.
	\end{enumerate}
\end{proof}
Now we can prove Proposition \ref{prop:anty.make.cont}
\begin{proof}[Proof of Proposition \ref{prop:anty.make.cont}]
Let us first give an overview of the proof, and then give full details.
Let $K\subset \R^n$  be a compact that contains an open neighborhood of the support $\supp u$ of $u$.
Consider $u$ as a functional on the Fr\'echet space $C^\infty_K(\R^n)$. By definition $u$ can be continuously extended to $C^k _K(\R^n)$ for some finite $k \ge 0$.  Thus, for any compact $K'\subset \R^n$, $u$ can be continuously extended to $C^k _{K'}(\R^n)$. Therefore, Lemma \ref{lem:anty.dual.make.L1.sm}\eqref{lem:anty.dual.make.L1.sm:1} implies that $A^k  u$ can be continuously extended to $C^0_{c}(\R^n)$, when $k$ is large enough. In other words, $A^k  u$ is a measure on $\R^n$. Applying Lemma \ref{lem:anty.dual.make.L1.sm}(\ref{lem:anty.dual.make.L1.sm:2},\ref{lem:anty.dual.make.L1.sm:3}) we obtain that $A^{k +2} u$ can be continuously extended to $L^1_{c}(\R^n)$. This implies that $A^{k +2} u$ is an absolutely continuous measure, and thus by the Radon-Nikodim Theorem it is given by a locally $L^1$ function (which has moreover compact support).
Thus, by Lemma \ref{cor:anty.make.L1.cont}, $A^{k +4} u$ is a continuous function.

With full details, the proof goes as follows.
	\begin{enumerate}
		\item As we assume $u$ to be a distribution with compact support, $u$ has a finite order $k$ by \cite[\S 2.3]{Hormander83}.
		\item By \cite[Theorem 2.1.6]{Hormander83}, there is a unique extension to $C^k_{c}(\RR^n)$. This extension is given in the following way. As in the proof of Proposition \ref{prop:gen}, for any $\phi \in C^k_{c}(\RR^n)$, we consider for any $\eps>0$, the function
	$$ \Psi_{\eps}:x \mapsto \frac{1}{\eps^n}\Psi\left( \frac{x}{\eps} \right).$$
	The function $\phi_\eps:=\Psi_{\eps}*\phi$ belongs to $C_{c}^{\infty}(\RR^n)$ and we have the convergence $\Psi_\eps * \phi \to \phi$ for the Fr\'echet space $C^k_{c}(\RR^n)$, meaning $\sup \left|\partial^{\alpha}\phi - \partial^{\alpha}\phi_\eps\right| \underset{\eps \to 0}{\to} 0$ for $\left|\alpha\right|\leq k$. Thus, one has
	$$\la u,\phi \ra := \lim_{\eps \to 0} \la u,\phi_\eps \ra.$$
	As in the proof of Proposition \ref{prop:gen}, one has that $\la u,\phi_\eps \ra$ equals $\la u,\Psi_{\eps}*\phi\ra$ which furthermore equals
$$\la u\otimes \phi,(y,t)\mapsto 1_{\Supp{\phi}}(t)\Psi_{\eps}(y-t)\ra = \la \phi,t\mapsto \la u,y\mapsto \Psi_{\eps}(y-t)\ra \ra.$$
Thus, $\la u,\phi \ra$ equals
\begin{equation}\label{eq:lim-int}
\lim_{\eps \to 0} \int_{\RR^n}\phi(x)\la u,y\mapsto \Psi_{\eps}(y-x) \ra \mathrm{d}x =
\int_{\RR^n}\phi(x)\left(\lim_{\eps \to 0} \la u,y\mapsto \Psi_{\eps}(y-x)\ra \right)\mathrm{d}x.
\end{equation}

\item  By Lemma \ref{lem:anty.dual.make.L1.sm}\eqref{lem:anty.dual.make.L1.sm:1}, $A^k u$ can be continuously extended to $C^0_{c}(\RR^n)$, and so, $A^k u$ is a measure.
\item  By Lemma \ref{lem:anty.dual.make.L1.sm}(\ref{lem:anty.dual.make.L1.sm:2},\ref{lem:anty.dual.make.L1.sm:3}), $A^{k+2}u$ can be continuously extended to $L^1_c(\RR^n)$ and $A^{k+2}u$ is still a measure.
\item For any measurable set $E$ of measure 0 for the Lebesgue measure, the characteristic function $1_E$ is equal to 0 in $L^1_c(\RR^n)$, so $(A^{k+2}u)(1_E)=0$, and thus $E$ has also zero measure for the measure $A^{k+2}u$. Thus, $A^{k+2}u$ is absolutely continuous against the Lebesgue measure, and then by the Radon-Nikodim theorem it is locally given by an $L^1$ function, with compact support.
\item By Lemma \ref{cor:anty.make.L1.cont}, $A^{k+4}u$ is a continuous function.
\end{enumerate}
This concludes the proof of Proposition \ref{prop:anty.make.cont}. 
\end{proof}

\begin{remark}
Note that the argument given for Proposition \ref{prop:anty.make.cont} is far from being sharp, and the $+4$ can probably be improved using analysis of Sobolev spaces, but it is completely unimportant for our purposes.
\end{remark}

We are now ready to give the proofs of Theorem \ref{thm:derivative} and Corollary \ref{thm:gen.sm}.

\begin{proof}[Proof of Theorem \ref{thm:derivative}]
Follows directly from Propositions \ref{prop:anti-Cexp} and \ref{prop:anty.make.cont} and item (\ref{item:1:anty}) of Lemma \ref{lem:anty.char}.
\end{proof}

\begin{proof}[Proof of Corollary \ref{thm:gen.sm}]
It is enough to prove the statement for $u$ with compact support. So, suppose that $u$ is a $\cCexp$-class distribution on an open set $U\subset \RR^n$ and that $u$ has compact support. Consider a continuous function $g$ of $\cCexp$-class, as given by Theorem \ref{thm:derivative} (so that a certain iterated partial derivative of $g$ equals $u$). 
By Proposition \ref{prop:loci}, $g$ is almost everywhere analytic and, furthermore, there is a subanalytic dense open subset $U'$ of $U$ on which $g$ is analytic.  Hence, also $u$ is analytic on $U'$ (as a finite iteration of partial derivatives of $g$).
\end{proof}

\section{A resolution result for subanalytic sets and functions}\label{sec:resol}


Before stating our resolution result we slightly generalize some basic notions.

\begin{defn}\label{def:unit:mon}
Call a function $f:A\subset \RR^n\to\RR^k$ analytic if there is an open set $U\subset \RR^n$ containing $A$ and an analytic function $f_U$ on $U$ whose restriction to $A$ is $f$.

Call a function $u:A\subset \RR^n \to \RR$ an analytic unit if there is an open set $U\subset \RR^n$ containing $A$ and a non-vanishing analytic function $u_U$ on $U$ whose restriction to $A$ is $u$. 


By a monomial we mean a polynomial of the form $d\cdot\prod_{i=1}^m x_i^{\mu_i}$ for some integers  $m\ge 0$ and $\mu_i\geq 0$ and some  $d\in \RR$.

Say that a subanalytic set $X$ has pure dimension $m$, if it has local dimension $m$ at each $x\in X$.
\end{defn}


Note that for a compact subanalytic set $A\subset \RR^n$, an analytic unit on $A$ is automatically subanalytic. 



The following is our main resolution (or rather, alteration) result, for subanalytic functions. It has the flavor of an alteration result since it allows in particular transformations by power maps $x\to x^r$ for integers $r>0$, in each of the coordinates.

\begin{theorem}[Resolution result for  subanalytic  sets and functions]\label{thm:res:def}
Let $X\subset [0,1]^n$ be a closed
subanalytic set of pure dimension $m$ and let $f:X\to [0,1]^k$ be a subanalytic function for some $k\ge 0$ and $n\geq m\geq 0$.
Then there exist finitely many open subsets $U_i\subset X$ and
subanalytic, analytic, proper functions 
$$
\phi_i: [0,1]^m\to X
$$
such that the following hold for each $i$ 
\begin{enumerate}

\item\label{thm:res:def:1}
The set $U_i$
is a real analytic submanifold of $\RR^n$ of dimension $m$, and
 $\phi_i$ restricts to an analytic isomorphism (of analytic manifolds) between $(0,1)^m$ and $U_i$. 

\item\label{thm:res:def:4} The $U_i$ are pairwise disjoint, and the union of the $U_i$ is open dense in $X$.

\item\label{thm:res:def:6} If $m=n$, then there is an analytic unit $u_i$ on $[0,1]^m$ and a nonzero monomial $M_i$
such that
\begin{equation}\label{fjvarphii}
\Jac(\phi_i)(x) = u_i(x) M_i(x) \mbox{ for each $x$ in $(0,1)^m$,} 
\end{equation}
where $\Jac(\phi_i)(x)$ stands for the Jacobian of $\phi_i$ at $x$. 

\item\label{thm:res:def:3} There are analytic units $u_{ij}$ on $[0,1]^m$ and monomials $M_{ij}$
such that for each component $f_j$ of $f$ one has
\begin{equation}\label{fjvarphii2}
f_j (\phi_i (x)) = u_{ij}(x) M_{ij}(x) \mbox{ for each $x$ in $(0,1)^m$}, 
\end{equation}
with furthermore
$u_{i1}=1$ or $M_{i1}=1$. 
\end{enumerate}
\end{theorem}

\begin{remark}
The final property of (\ref{thm:res:def:3}) about $f_1(\phi_i)$ being equal to $u_{i1}$ or $M_{i1}$ will be used for the reduction to the base case of Proposition \ref{cor:base-case-u} in the proof of Theorem A for locally integrable $u$. \change{Also note that the monomials in Theorem \ref{thm:res:def} all have natural exponents by our Definition \ref{def:unit:mon}.}
\end{remark}

In fact, there are two natural proof strategies for Theorem \ref{thm:res:def}. One uses the Uniformization Theorem (4.19) from \cite{DvdD} and Hironaka's 's embedded resolution of singularities from  \cite{Hir:Res}, and the description of subanalytic functions by terms as in Corollary 2.5 of \cite{DriesMacinMar}. A second proof uses the Lion-Rolin,  Parusi\'nski,  Miller preparation result for subanalytic functions \cite{LR} \cite{Paru2}, \cite{MillerD} and the rectilinearization result in the form of Theorem 1.5 of \cite{CluckersMiller03}, which is based on the mentioned preparation result. As the first strategy via Hironaka's resolution resembles very much the proof in \cite{AizC} for the corresponding $p$-adic resolution result,
we will implement the second proof strategy. 

%
%

We use the following corollary of Theorem 1.5 of \cite{CluckersMiller03}. 
\begin{prop}\label{propCM}
Let $f:X\to \RR^k$ be a subanalytic function on a subanalytic set $X\subseteq\RR^n$ for some $k\ge 0$ and $n\ge 0$.  Then there exists a finite partition $\cU$ of $X$ into subanalytic, analytic submanifolds of $\RR^n$ such that for each $U\in\cU$, there exist $d\in\{0,\ldots,n\}$ and a subanalytic, analytic isomorphism $\phi:(0,1)^d\to U$, such that for each function $g$ in the set $\G$ defined by
\begin{equation}\label{eq:G}
\G = \begin{cases}
\{f_i\circ \phi\mid i=1,\ldots,k\}
    & \text{if $d < n$,} \\
\{f_i\circ \phi\mid i=1,\ldots,k\} \cup\{\Jac(\phi)\},
    & \text{if $d = n$,} \\
\end{cases}
\end{equation}
the function $g$ is either identically zero on $(0,1)^d$ or it may be written in the form
\[
g(z) = \left(\prod_{j=1}^{d} z_{j}^{r_j}\right) u(z)
\]
for $z$ in $(0,1)^d$, for some integers $r_j$ and an analytic unit $u$ on $[0,1]^d$.
\end{prop}

In the above proposition and theorem, 
each of $(0,1)^d$, $[0,1]^d$ and $\RR^d$ stand for $\{0\}$ when $d=0$.  

\begin{proof}[Proof of Proposition \ref{propCM}]
	This is Theorem 1.5 of \cite{CluckersMiller03} in the special case with $m=0$, with the extra conclusion that one can take $l=0$
when $m=0$ in Theorem 1.5 of \cite{CluckersMiller03}. Let us explain this. By Theorem 1.5 of \cite{CluckersMiller03} with $m=0$, we find a partition $\cU_0$ of $X$ into analytic, subanalytic submanifolds of $\RR^n$, such that for each $U_0$ in $\cU_0$ there are $d\ge \ell\ge 0$, an open analytic, subanalytic cell $B\subset (0,1)^\ell$
and a subanalytic, analytic isomorphism of analytic manifolds
$$
\phi_0:B\times (0,1)^{d-\ell}\to U_0,
$$
with the following properties. For each function $g_0$ in the set $\G_0$ defined by
\begin{equation}\label{eq:G1}
\G_0 = \begin{cases}
\{f_i\circ \phi_0\mid i=1,\ldots,k\}
    & \text{if $d < n$,} \\
\{f_i\circ \phi_0\mid i=1,\ldots,k\} \cup\{\Jac(\phi_0)\},
    & \text{if $d = n$,} \\
\end{cases}
\end{equation}
the function $g_0$ is either identically zero on $B\times (0,1)^{d-\ell}$ or it may be written in the form
\[
g(z) = \left(\prod_{j=\ell+1}^{d} z_{j}^{r_j}\right) u(z)
\]
for $z$ in $(0,1)^d$, for some rational numbers $r_j$ and an analytic unit $u$ on $\overline B\times [0,1]^{d-\ell}$, where $\overline B$ is the closure of $B$ in $\RR^\ell$. By composing further with a suitable power map
$$
(z_i)_i\mapsto (z_i^r)_i
$$
in each coordinate for some suitable integer $r>0$, we may suppose that the $r_j$ are integers. Furthermore, by induction on $n$, we may suppose that each of the cell walls of $B$ is also of that form, in particular, each of them is either identically zero or an analytic unit on $\overline B$. (For the notion of walls of cells, see Definition 4.2.5 of \cite{CPW}.) To fix notation, say that the variable $z_i$ has cell walls $\alpha_i(z_1,\ldots,z_{i-1})$ and $\beta_i(z_1,\ldots,z_{i-1})$, for $i=1,\ldots,\ell$, with $\alpha_i<\beta_i$. Now compose with one more transformation $\phi_1$, sending $z$ in $(0,1)^d$ to the tuple $x$ with first component $x_1 = (\beta_1-\alpha_1)z_1 + \alpha_1$, second component $x_2 = (\beta_2(x_1)-\alpha_2(x_1))z_2 + \alpha_2(x_1)$ and so on until the $\ell$-th component $x_\ell$, and then the identity $x_{\ell+i}$ for $i>0$ up to $i=d - \ell$. Now the so-obtained maps $\phi=\phi_0\circ \phi_1$ are as desired.
\end{proof}

\begin{proof}[Proof of Theorem \ref{thm:res:def}]
	All properties of the theorem except that one can take $u_{i1}=1$ or $M_{i1}=1$ for each $i$ as in (\ref{thm:res:def:3}) follow directly from Proposition \ref{propCM} from the extra conditions in Theorem \ref{thm:res:def} on $X$ and $f$ and some standard o-minimal properties. Indeed, one applies Proposition \ref{propCM} to $X$ and $f$ and one retains only the pieces of full dimension $m$ to get the correct conclusion right away. \change{In particular, one notes that the occurring integer exponents in Proposition \ref{propCM} have to be non-negative under the extra conditions of Theorem \ref{thm:res:def}, by the boundedness of the ranges and by the change of variables formula when $m=n$.} To achieve the remaining property
about $u_{i1}=1$ or $M_{i1}=1$ from the end of (\ref{thm:res:def:3}), fix $U_i$ and $\phi_i$ and assume that $M_{i1}$ is a non-constant monomial $d\prod_{i=1}^m x_i^{a_i}$, with integers $a_i\ge 0$, $a_m>0$ and nonzero $d$ in $\R$. We do induction on the number of $a_i$ with $a_i>0$.
Consider $w_i(x) := (u_{i1}(x))^{1/a_m}$ for $x$ in $[0,1]^m$. Then $w_i$ is also an analytic unit on $[0,1]^m$. Now consider the map
$$
\psi_i:[0,1]^m \to \RR^m :  x \mapsto (x_1 ,\ldots,x_{m-1},x_m w_i(x)).
$$

Call an open subanalytic cell $B$ in $(0,1)^m$ special if each cell wall is either constant, or, the restriction of an analytic unit on the closure $\overline B$ of $B$.
 By compactness of $[0,1]^m$ and the inverse function theorem, there is a finite collection of mutually disjoint special cells $B_{\ell}$ in $(0,1)^m$ such that
$$
\bigcup_{\ell}\overline B_\ell = [0,1]^m
$$
and such that for each $\ell$, either one has $x_m\not=0$ on $\overline B_\ell$, or, the map $\psi_i$ restricts to an analytic diffeomorphism on $\overline B_\ell\to \overline C_\ell$ onto $\overline C_\ell$, with $C_\ell\subset \RR^m$ an open box (namely, a Cartesian product of open intervals).  We can now easily finish. If $x_m\not=0$ holds on $\overline B_\ell$, then $x_m^{a_m}$ is an analytic unit on $\overline B_\ell$, and thus, we are done by induction on the number of $a_i$ with $a_i>0$ by transforming the cell $B_\ell$ into $(0,1)^m$ using the cell walls as in the proof of Proposition \ref{propCM}.  In the other case that $\psi_i$ restricts to an analytic diffeomorphism from $\overline B_\ell$ onto $\overline C_\ell$, one falls in a situation with $u_1=1$ after replacing $\phi_i$ by $\phi_i\circ \phi_{i\ell}\circ r_\ell$, where $\phi_{i\ell}$ is the inverse of the restriction of $\psi_i$ to $\overline B_\ell\to \overline C_\ell$ and $r_\ell:[0,1]^m\to \overline C_\ell$ the obvious bijection of the form $x\mapsto (r_ix_i+b_i)_i$ for some reals $r_i>0$ and $b_i\ge 0$ for $i=1,\ldots,m$.
\end{proof}

Essentially combining Theorem \ref{thm:res:def} with Proposition \ref{prop:mult}, we obtain the following result.

\begin{prop}\label{lem:u:int}
Let $u$ be a distribution on an open set $U$ in $\RR^n$ which is locally $L^1$. Let $g$ be a subanalytic function on $U$ which is nowhere vanishing and such that $gu$ is also locally $L^1$. Then $u$ is holonomic if and only if $g u$ is holonomic.
\end{prop}
\begin{proof}
As this is a local question, we may focus on an open subset $U'$ of $U$, by working with a covering of $U$ by open subsets. 
Up to translating and shrinking, we may suppose that $U'$ is an open subanalytic subset of $(0,1)^n$ such that $U$ contains the closure $\overline U'$ of $U'$ in $\R^n$, and that $g^\delta$ takes values in $(0,1]$ for $\delta$ being either $1$ or $-1$. Apply Theorem \ref{thm:res:def} to the function $f$ which is the restriction of $g^\delta$ to the set $X:=\overline U'$, yielding maps $\phi_i$ and open sets $U_i$. Hence, up to writing $u$ as the finite sum of the restrictions $u_{|\overline U_i}$ and working with each restriction separately (which is allowed by Proposition \ref{prop:sum:holon}), it is enough to prove the result for each $u_{|\overline U_i}$.
Hence,
by the change of variables formula (used as in the proof of Proposition \ref{prop:power-map:holon} but now with transformation $\phi_i$)
and Theorem \ref{thm:push}, we may suppose that $U_i=(0,1)^n$, and that $g^\delta_{|[0,1]^n}$ equals an analytic unit on $[0,1]^n$ times a monomial. But then the corollary follows from Propositions \ref{prop:product:unit}, \ref{prop:mult} and \ref{prop:mult:half}.
\end{proof}

\section{Holonomicity of $\cCexp$-class distributions}\label{sec:holon}

In this section we first prove Theorem \ref{mainthm} under the additional assumption that $u$ is a locally integrable function.
The general theorem will then follow from this case by our results on anti-derivatives to reduce to the locally integrable case.



Consider a distribution $u$ on an open \blue{subanalytic} set $U\subset \RR^n$. Suppose that $u$ is given by a locally $L^1$ function $f$ on $U$, and that $f$ is of $\Cexp$-class.
We first show that such $u$ is holonomic.

\begin{proof}[Proof of Theorem \ref{mainthm} for locally integrable distributions]
Since holonomicity is a local property on $U$, we may suppose that the support of $f$ is included in $[0,1]^n$ and that $[0,1]^n\subset U$, with notation from just above the proof.
By Lemma \ref{module}, 
we can write $f$ on $U$ as
$$
f(x) = \int_{t\in\RR} H (x,t) \exp(\mathrm{i}t) \mathrm{d}t
$$
where $H$ lies in  $\cC(U\times \RR)$, 
and, for each $x$ in $U$, the function $t\mapsto H(x,t)$ is  $L^1$ on $\RR$. The main difficulty for this proof  is the fact that, in general, $H$ is not locally $L^1$ as a function on $U\times \RR$, hence, by itself, it does not automatically give a distribution on $U\times \RR$.
However, by Propositions \ref{prop:Cexp:int} and \ref{lem:u:int}, we may suppose that $H$ is  $L^1$ on $U\times \R$ (namely, up to replacing $H$ by $Hg$ with $g$ from the final part of Proposition \ref{prop:Cexp:int}, which is harmless by Proposition \ref{lem:u:int}).
Let $\cG_0$ be the collection of \blue{$(x,t)\mapsto t$ and} subanalytic functions appearing in the build-up of $H$ (that is, we write $H$ as a finite sum of products of some of the subanalytic functions in $\cG_0$ and of logarithms of some of the positive subanalytic functions in $\cG_0$, which is possible by the definition of $\cC$-class functions). Up to slightly rewriting and \change{finite partitioning}, we may suppose that each $g$ in $\cG_0$ has constant sign, and either takes values outside $[-1,1]$, or, inside $[-1,1]$. Let $\cG$ be the corresponding collection of functions $\pm g^{\pm 1}$ for $g$ in $\cG_0$ such that all functions in $\cG$ take values in $[0,1]$.
Now apply the resolution result of Theorem \ref{thm:res:def} (in the $(x,t)$-space) to the map whose component functions are the functions in $\cG$ and with the function $(x,t)\mapsto \delta_1 t^{\delta_2}$ in the role of the component function $f_1$ in item (\ref{thm:res:def:3}) of Theorem \ref{thm:res:def} \change{and where the $\delta_i$ are $1$ or $-1$ so that $(x,t)\mapsto \delta_1 t^{\delta_2}$ appears in $\cG$}. Up to a change of variables (as in the proofs of Propositions \ref{prop:power-map:holon} and \ref{lem:u:int}, with transformation \change{one of the maps $\phi_i$ from Theorem \ref{thm:res:def}}), we are now in the base case of Proposition \ref{cor:base-case-u} possibly multiplied with an analytic unit, \change{or more precisely, a finite sum of those situations}. (Here, we have used the final statement of (\ref{thm:res:def:3}) of Theorem \ref{thm:res:def} for \change{$\delta_1 t^{\delta_2}$}, in the role of $f_1$.) \change{Let us explain into more detail how the various
logarithms of monomialized subanalytic functions (as in Theorem \ref{thm:res:def}) need to be expanded, after we have rewritten by the change of variables formula. The main idea is that if we have a factor of the form $\log ( x^\alpha u(x) )$ for $x$ in $(0,1)^{n+1}$ and for some $\alpha\in\ZZ^{n+1}$ and positively-valued analytic unit $u$ on $[0,1]^{n+1}$, we can choose a sufficiently large constant $c>1$ so that $\log (cu(x))$ is an analytic unit on $[0,1]^{n+1}$, and then we can expand as follows:
$$
\log ( x^\alpha u(x) ) = \log (c^{-1}) + \sum_{i=1}^{n+1} \alpha_i \log  (x_i) + \log (cu(x)).
$$
Also after our use of the change of variables formula, the oscillatory factor now may have one of the following three forms, where again $x$ runs over $(0,1)^{n+1}$
\begin{itemize}
\item $e^{\pm \mathrm{  i}}$,

\item $e^{\mathrm{i} x^\alpha}$ for some $\alpha \in \NN^{n+1}$,

\item $e^{\mathrm{i} x^{-\alpha}}$ for some $\alpha \in \NN^{n+1}$.
\end{itemize}
In the first two cases this factor is a complex-valued analytic unit on $[0,1]^{n+1}$ (that is, the real part and the imaginary part are analytic units), and in the third case it is of the same form as the oscillatory factor in Proposition \ref{cor:base-case-u}.
By propositions \ref{cor:base-case-u} and \ref{prop:product:unit} and by Theorem \ref{prop:push:holon} it follows that, for each choice of signs $\delta_i=\pm 1$, the distribution $u_\delta$ on $U\times \RR$ given by the locally integrable function $G(x,t)$ sending $(x,t)$ to $J(t)  H (x,t) \exp(\delta_1\mathrm{i}t^{\delta_2})$ when $t$ lies in $(0,1)$ and to $0$ otherwise is holonomic, where $J(t)$ is the Jacobian of the map $t\mapsto \delta_1t^{\delta_2}$.
Theorem \ref{prop:push:holon} now finishes the proof of the special case of Theorem \ref{mainthm} for locally integrable distributions of $\cCexp$-class, where we apply it to the projection map $U\times \RR\to U$ which is proper on the support of $u_\delta$. Indeed, our given distribution $u$ is the finite sum of the push-forwards of the $u_\delta$, so we are done.}
\end{proof}



We can now prove Theorem \ref{mainthm} in the general case, by reducing to the locally integrable case using the results from Section \ref{sec:antider} on anti-derivatives.

\begin{proof}[Proof of Theorem \ref{mainthm}, general case]
	By using \blue{anti-derivatives} as in Proposition \ref{prop:anty.make.cont} and by Propositions \ref{prop:Cexp.and.cont.is.CexpFun}, \ref{prop:anti} and \ref{prop:anti-Cexp}, we directly reduce to the already treated case of locally integrable distributions.
To use Proposition \ref{prop:anty.make.cont} we need a distribution with compact support. As holonomicity is local property, for any $x_0 \in U$, we can consider a smooth $\cCexp$-class function $\theta$ with compact support and equal to 1 on a open neighborhood of $x_0$, and work with $\theta u$ instead of $u$.
\end{proof}



%
%
%


\subsection{Generalizations}\label{sec:gen}
All the results of the present paper extend to the more general framework of  ${\mathcal {C}}^{\CC,\cF}$-class functions (see \cite[Definition 2.7]{CCS}). In Lemma \ref{prop:h1h2} we showed that the class $\cCexp$ can be built from the class of constructible functions by applying an integral transform with kernel $t\mapsto \exp(\text{i}t)$. Similarly, the class ${\mathcal {C}}^{\CC,\cF}$ is constructed by applying this same integral transform to functions which are not just constructible, but power-constructible (the latter class  ${\mathcal {C}}^\CC$, defined in \cite[Definition 2.2]{CCRS}, is the collection of $\mathbb{C}$-algebras of functions defined on subanalytic sets, generated by complex powers and logarithms of subanalytic functions). In order to extend the results of this work to the class ${\mathcal {C}}^{\CC,\cF}$, we replace Theorem \ref{thm:stab} by \cite[Theorem 2.9]{CCS}, Proposition \ref{prop:limits} by \cite[Theorem 7.11]{CCS} and we modify the proof of Lemma \ref{lem:H:int} in the following way: the subanalytic function $f_i$ are now raised to some complex power $\alpha_i$, whose real part we may suppose to be nonnegative; the factor $(1+|f_i|)$ appearing in the proof should now be replaced by $(1+|f_i|^N)$, where $N$ is some natural number bigger than the real parts of all the exponents $\alpha_i$. All the proofs of this paper go through unchanged, up to replacing every occurrence of the word “constructible” by “power-constructible” (note in particular that the Base case, treated in Lemma \ref{prop:base-case} and Proposition \ref{cor:base-case-u}, already takes into account complex exponents $a_j$).


\subsection{Open questions}
Let us end with a few open questions on $\cCexp$-class distributions. It would be interesting to know whether or not $\cCexp$-class distributions are automatically tempered, and whether or not they have regularisations of $\cCexp$-class. Furthermore, it would be interesting to know more about wave front sets of $\cCexp$-class distributions, for example, whether or not they are always equal to the complement of the zero locus of a $\cCexp$-class function, in analogy to the $p$-adic situation from Theorem 3.4.1 of \cite{CHLR}. Finally, the condition that the support of $u$ is compact in Theorem \ref{thm:derivative} and Corollary \ref{thm:gen.sm} is possibly not necessary.

\bibliographystyle{amsplain}
\bibliography{anbib}
\end{document}